\documentclass[aos,preprint]{imsart}
\setattribute{journal}{name}{}

\RequirePackage[OT1]{fontenc}
\RequirePackage{amsthm,amsmath}
\RequirePackage[authoryear]{natbib}
\RequirePackage[citecolor=blue,urlcolor=blue,breaklinks]{hyperref}

\RequirePackage{amsbsy,amssymb,color,graphicx,epstopdf}

\startlocaldefs
\newcommand{\I}{\ensuremath{\operatorname{\mathbb{I}}}}
\renewcommand{\P}{\ensuremath{\operatorname{\mathbb{P}}}}
\newcommand{\E}{\ensuremath{\operatorname{\mathbb{E}}}}
\newcommand{\var}{\ensuremath{\operatorname{Var}}}
\newcommand{\cov}{\ensuremath{\operatorname{Cov}}}
\makeatletter
\newcommand{\inlineoverset}[2]{%
  {\mathop{#2}\limits^{\vbox to 2\ex@{\kern-\tw@\ex@
   \hbox{\scriptsize \ensuremath{#1}}\vss}}}}
\makeatother
\numberwithin{equation}{section}
\theoremstyle{plain}
\newtheorem{theorem}{Theorem}[section]
\newtheorem{lemma}{Lemma}[section]
\newtheorem{proposition}{Proposition}[section]
\newtheorem*{corollary}{Corollary}
\theoremstyle{remark}

\makeatletter
\def\myl{\mathopen\mybig}
\def\myr{\mathclose\mybig}
\def\mybigx#1{\dimen@#1\relax
\mathchoice
{\vbox to \dimen@{}}%
{\vbox to \dimen@{}}%
{\vbox to .7\dimen@{}}%
{\vbox to .5\dimen@{}}}%
\def\mybig#1{{\hbox{$\left#1\mybigx{0.8em}\right.\n@space$}}}
\makeatother
\endlocaldefs

\begin{document}

\begin{frontmatter}

\title{Degree-based network models}
\runtitle{Degree-based network models}

\begin{aug}
\author{\fnms{Sofia C.} \snm{Olhede}\ead[label=e1]{s.olhede@ucl.ac.uk}}
\and
\author{\fnms{Patrick J.} \snm{Wolfe}\ead[label=e2]{p.wolfe@ucl.ac.uk}}

\runauthor{Olhede \& Wolfe}

\affiliation{University College London}

\address{Department of Statistical Science\\ University College London\\ Gower Street \\London WC1E 6BT, UK\\ \printead{e1,e2}}

\end{aug}

\begin{abstract}
We derive the sampling properties of random networks based on weights whose pairwise products parameterize independent Bernoulli trials. This enables an understanding of many degree-based network models, in which the structure of realized networks is governed by properties of their degree sequences. We provide exact results and large-sample approximations for power-law networks and other more general forms. This enables us to quantify sampling variability both within and across network populations, and to characterize the limiting extremes of variation achievable through such models. Our results highlight that variation explained through expected degree structure need not be attributed to more complicated generative mechanisms.
\end{abstract}



\end{frontmatter}

\section{Introduction}

Mechanisms that generate networks have lately become the focus of considerable study in statistical methodology \citep{bickel2009nonparametric, rohe2011spectral, BickelLevina, fienberg2012maximum, zhao2012consistency, rinaldo2013maximum, sussman2013universally}. Networks are naturally summarized through their degrees. These count each node's neighbors, and hence reflect the observed proclivity of nodes to participate in network connections. Networks generated from statistical models or probabilistic constructions exhibit variation in observed degree values, relative to their population expectations.

Quantifying this variation remains an important open problem. To address it, we adopt a model in which the structure of realized networks is determined by their degrees, so that no confounding effects risk compounding the observed variation. A natural first approach \citep{newman2002random} is to fix or sample a vector $\underline{d}$ of $n$ counts, and then select uniformly at random from the set of all $n$-node graphs that admit $\underline{d}$ as their degree sequence. This theoretical construct is termed the configuration model. In practice, edges are sequentially assigned to nodes in ways that respect $\underline{d}$ akin to sampling with replacement \citep{molloy1995critical}. This construction is asymptotically valid, but it can assign an edge to a single node---yielding a self-loop---or multiple edges to the same pair of nodes \citep{Durrett}. For any network of finite size, these self-loops and multiple edges mean that the exactness of the likelihood interpretation is lost.

An important modification relaxes the requirement of a specified $\underline{d}$, replacing it with a set of $n$ parameters describing the propensity of each node to form connections. Edges are then modeled as independent Bernoulli trials with success probabilities $p_{ij}$ for all $1 \leq i \leq j \leq n$. \citet{Chung2002pnas} associated a nonnegative weight $w_i$ to each node $i$ and then set $p_{ij} = w_i w_j / \|\underline{w}\|_1$, where $\|\underline{w}\|_1 = \sum_{k=1}^n w_k$. Under the normalization constraint that $w_i^2 \leq \|\underline{w}\|_1$ for all $i$, it follows that the expected degree $\E(d_i)$ of the $i$th node is equal to $w_i$, and thus the unnormalized weights $w_1, \ldots w_n$ can be interpreted as expected degrees. This model is immediately recognizable as the log-linear model $\log p_{ij} = \log w_i + \log w_j - \log \|\underline{w}\|_1$, connecting to statistical methods such as maximum likelihood estimation of $p_{ij}$ from an observed network \citep{holland81exponential, fienberg2012maximum, PerryWolfe2012a}. It is also a special case of assigning edge probabilities $p_{ij}$ using a kernel \citep[Section~16.4]{bollobas2007phase}. \citet{janson2010asymptotic} and \citet{van2012critical} derive limiting properties of networks that result from using a kernel to assign edge probabilities.

Further statistical insight has been given by \citet{chatterjee2011random}, who among others \citep{britton2006generating, bollobas2007phase} recognized that when a logistic-linear model is used, then all graphs with the same degree sequence have equal likelihood, with the degrees a sufficient statistic. Though this sacrifices the rank-one multiplicative model structure of \citet{Chung2002pnas}, \citet{janson2010asymptotic} gives conditions under which these two models are asymptotically equivalent as $n \rightarrow \infty$, and \citet{PerryWolfe2012a} give finite-sample conditions under which near-optimal likelihood-based estimates for both models can be obtained directly as $\hat{p}_{ij} = d_i d_j / \|\underline{d}\|_1$. We later give a limit theorem for this estimator, which sees wide use in practice \citep{bickel2009nonparametric}.

\section{Modeling the degrees of a network}
\label{sec:degreeModels}

A simple random graph on $n$ nodes is represented by a symmetric adjacency matrix as
\begin{equation}
\label{eq:BernoulliModel}
A_{ij} \sim \operatorname{Bernoulli}(p_{ij}) \quad (1 \leq i < j \leq n); \quad A_{ji} = A_{ij}, \,\, A_{ii} = 0.
\end{equation}
Simple graphs are undirected, with neither self-loops nor multiple edges. We model their edges as independent, giving rise to a data log-likelihood $ \sum_{i<j} \left\{ A_{ij} \log(p_{ij}) + (1-A_{ij})\log(1-p_{ij}) \right\}$.

Each network degree is thus a sum $d_i = \sum_{j \neq i} A_{ij}$ of $n - 1$ independent $\operatorname{Bernoulli}(p_{ij})$ variates. The probability that $d_i$ takes a given value $k$ is the sum of all distinct ways in which $k$ successes can occur in $n-1$ Bernoulli trials. We count these $\binom{n-1}{k}$ ways through index sets $S_i^l$, and complements $\bar{S}_i^l = \{1,\ldots, n\} \setminus S_i^l$, yielding a Poisson--Binomial distribution for each $d_i$:
\begin{equation}
\label{eq:PoissonBinomial}
\P\big( d_i \!=\! k \,\vert\, \{p_{ij}\}_{j\neq i} \big) = \!\! \sum_{l=1}^{\binom{n-1}{k}} \! \Bigg\{ \! \prod_{j' \in S_i^l} p_{ij'} \!\Bigg\} \Bigg\{\! \prod_{\bar{\jmath}' \in \bar{S}_i^l} \!(1-p_{i\bar{\jmath}'}) \!\Bigg\} \quad (0 \leq k \leq n - 1).\!\!\!\!
\end{equation}

To specify this model fully requires $\binom{n}{2}$ parameters, each corresponding to a particular $p_{ij}$. This is termed an inhomogeneous random graph. Setting all $p_{ij} = p$ recovers the homogeneous random graph model studied by Erd\H{o}s, R\'enyi, and others, and~\eqref{eq:PoissonBinomial} then reduces to the $\operatorname{Binomial}(n-1,p)$ distribution for all $i$. In between these two extremes lie other parameterizations with controlled variability, such as the models described in the Introduction.

\subsection{A multiplicative model for the probability of linkage}\label{sec:multModel}

The choice of model for $p_{ij}$ determines properties of the Poisson--Binomial distribution of~\eqref{eq:PoissonBinomial}, such as its variance. This distribution describes the variation of a fixed node $i$ across multiple network realizations. We choose to parameterize the multiplicative model of \citet{Chung2002pnas}, described in the Introduction, via a normalized weight vector $\underline{\pi}$ of length $n$:
\begin{equation}
\label{eq:PiMultModel}
p_{ij} = \pi_i \pi_j \quad (1 \leq i < j \leq n); \quad \pi_i \in \Pi_n \subseteq [0,1] \quad (1 \leq i \leq n).
\end{equation}

This parameterization has two important consequences for our statistical understanding of degree-based network models. First, it decouples the edge generation probabilities $p_{ij}$, so that each $p_{ij}$ depends only on two parameters, and dispenses with the need to artificially constrain $w_i^2 \leq \|\underline{w}\|_1$ for all $i$ in the parameterization $p_{ij} = w_i w_j / \|\underline{w}\|_1$ of \citet{Chung2002pnas}. Second, by allowing the range $\Pi_n$ of each normalized weight in~\eqref{eq:PiMultModel} to shrink as $n$ increases, we obtain more realistic asymptotic regimes and large-sample properties of sequences of networks; see Section~\ref{sec:scaling}. The following proposition is a direct consequence of the model specified by~\eqref{eq:BernoulliModel}--\eqref{eq:PiMultModel}.

\begin{proposition}[Conditional degree characteristics]\label{thm:condModels}
Let $\underline{\pi} \in [0,1]^n$ be a deterministic vector of parameters, and let $\underline{d}$ be the degree vector of an $n$-node simple graph whose edges are independent $\operatorname{Bernoulli}(\pi_i \pi_j)$ trials. Then for $0 \leq k \leq n-1$, $\P(d_i = k \,\vert\, \underline{\pi} )$ is given by~\eqref{eq:PoissonBinomial} with $p_{ij} = \pi_i \pi_j$, and
\begin{align}
\label{eq:condE} \E(d_i \,\vert\, \underline{\pi}) & = \sum_{j \neq i} \E(A_{ij} \,\vert\, \underline{\pi}) = \pi_i \sum_{j \neq i} \pi_j,
\\ \label{eq:condVar} \var(d_i \,\vert\, \underline{\pi}) & = \sum_{j \neq i} \var(A_{ij} \,\vert\, \underline{\pi}) = \E(d_i \,\vert\, \underline{\pi}) - \pi_i^2 \sum_{j \neq i} \pi_j^2,
\\ \label{eq:condCov} \cov(d_i, d_j \,\vert \, \underline{\pi}) & = \var(A_{ij} \,\vert\, \underline{\pi}) = \pi_i \pi_j (1 - \pi_i \pi_j) \quad (i \neq j, \,\, n \geq 3).
\end{align}
\end{proposition}

We see from Proposition~\ref{thm:condModels} that properties of the parameter vector $\underline{\pi}$ have strong and direct repercussions for the realized degrees of a given network. First, the expectation of degree $d_i$ behaves like $\pi_i$ scaled by $\| \underline{\pi} \|_1$. This norm, and thus $\E(d_i \,\vert\, \underline{\pi})$, may grow with the network size $n$. Second, $\var(d_i \,\vert\, \underline{\pi})$ behaves like $\pi_i$ scaled by $\| \underline{\pi} \|_1 - \pi_i \| \underline{\pi} \|_2^2$, which may also grow with $n$. Third, whenever $\var(d_i \,\vert\, \underline{\pi})$ is growing in $n$, the correlation between the $i$th degree and all others will decay toward zero. This may be seen directly from~\eqref{eq:condCov}, because the covariance between any two distinct degrees is bounded by $1/4$, the maximum variance of a Bernoulli trial.

Finally, from~\eqref{eq:condE} and~\eqref{eq:condVar} we obtain the dispersion of $d_i$ as
\begin{equation}
\label{eq:disp}
\frac{\var(d_i \,\vert\, \underline{\pi})}{\E(d_i \,\vert\, \underline{\pi})}
= 1 - \pi_i \frac{\|\underline{\pi}\|_2^2-\pi_i^2}{\|\underline{\pi}\|_1-\pi_i},
\end{equation}
defined whenever $\E(d_i \,\vert\, \underline{\pi}) > 0$. Thus $d_i$ is under-dispersed relative to a Poisson variate. This under-dispersion is controlled by $\pi_i$ directly, and by the remaining $n-1$ network parameters in aggregate. Specifically, whenever $\pi_i$ or the norm ratio $\sum_{j \neq i} \pi_j^2 / \sum_{j \neq i} \pi_j$ goes to zero with increasing $n$, the dispersion of $d_i$ is squeezed to $1$. Applying the Cauchy--Schwarz inequality $( \sum_{j \neq i} \pi_j )^2 \leq (n-1) \sum_{j \neq i} \pi_j^2$ to~\eqref{eq:disp} further quantifies this effect.

\begin{corollary}\label{thm:dispBnd}
The difference $1 - \var(d_i \,\vert\, \underline{\pi}) / \E(d_i \,\vert\, \underline{\pi})$ satisfies
\begin{equation}
\label{eq:dispDiff}
\frac{\E(d_i \,\vert\, \underline{\pi})}{n-1}
\leq 1 - \frac{\var(d_i \,\vert\, \underline{\pi})}{\E(d_i \,\vert\, \underline{\pi})}
\leq \frac{\E(d_i \,\vert\, \underline{\pi})}{\|\underline{\pi}\|_1-\pi_i} = \pi_i.
\end{equation}
\end{corollary}

The lower bound in~\eqref{eq:dispDiff} implies that if $\E(d_i \,\vert\, \underline{\pi})$ is of order $n$, then the dispersion of~\eqref{eq:disp} is bounded away from $1$ as $n \rightarrow \infty$, and so $d_i$ cannot become Poisson. To say more, we must choose a form for $\underline{\pi}$. Theorem~\ref{thm:powerLaws} later establishes that if $\pi_i \propto i^{-\gamma}$, then~\eqref{eq:disp} regulates the distribution of $d_i$.

\subsection{A multiplicative model with random weights}\label{hierarchical}

To obtain a heterogeneous population of networks, we may assume elements of $\underline{\pi}$ be random, in accordance with some law $F(\underline{\pi})$. This yields a hierarchical generative model, with $\P(d_i = k) = \int \P(d_i = k \,\vert\, \underline{\pi}) \, dF(\underline{\pi})$. If $\underline{\pi}$ is a random sample from some univariate $F(\pi)$, then the Bernoulli edge trials comprising a single degree $d_i$ will be exchangeable---they are conditionally independent and identically distributed given $\pi_i$---and all degrees will be identically distributed.

As discussed by \citet[p.~8--9]{bollobas2007phase}, there is a natural link between treating the elements of $\underline{\pi}$ as a random sample and viewing them as a deterministic, decaying sequence. A change of measure relates a random sample from $F(\pi)$ on $[0,1]$ to the uniform distribution on this interval. Since the expectations within an ordered uniform random sample go as $i/n$, the deterministic inverse law values $F^{-1}(i/n)$ can be directly related to properties of the random sample as $n$ grows large.

Let us explore the finite sampling effects of this choice more clearly.

\vspace{-0.2\baselineskip}%
\begin{proposition}[Marginal degree characteristics]\label{thm:hierModels}
Let $\underline{\pi}$ be a random sample from a probability law $F(\pi)$ on $[0,1]$ with mean $\mu$ and variance $\sigma^2$, and consider a simple $n$-node graph whose edges given $\underline{\pi}$ are independent $\operatorname{Bernoulli}(\pi_i \pi_j)$ trials. Then
\begin{align}
\label{eq:margED} \E(d) & = (n-1) \mu^2,
\\ \label{eq:margVar} \var(d) & = (n-1) \E(d) \left\{ \sigma^2 + \frac{1 - (\mu^2 + \sigma^2) }{n-1} \right\},
\\ \label{eq:margCov} \cov(d_i, d_j) & = \E(d) \left\{ \frac{3(n-2)}{n-1} \sigma^2 + \frac{1-(\mu^2 + \sigma^2)}{n-1} \right\} \,\,\, (i \neq j, \,\, n \geq 3),
\end{align}
and each degree is a Binomial mixture with mixing distribution $F(\pi / \mu)$:
\begin{equation}
\label{eq:margProbD} \P(d = k) = \int g_n(t,k) \textstyle \, dF\big( \frac{t}{\mu} \big) \quad (0 \leq k \leq n-1),
\end{equation}
where $g_n(t,k) = \binom{n-1}{k} \, t^k (1 - t)^{n-1-k}$ is the $\operatorname{Binomial}(n-1,t)$ kernel.
\end{proposition}

\begin{proof}
The moments of~\eqref{eq:margED}--\eqref{eq:margCov} follow by marginalizing their conditional counterparts in~\eqref{eq:condE}--\eqref{eq:condCov}, using the law of total covariance in the latter two cases. The form of~\eqref{eq:margProbD} is recognizable as de~Finetti's representation of a sum of exchangeable indicator variables \citep{diaconis1977finite}. Indeed, for the $i$th degree $d_i$, write $d_i = \sum_{j \neq i} A_{ij}$, with each edge $A_{ij} \,\vert\, \pi_i, \pi_j$ an independent $\operatorname{Bernoulli}(\pi_i \pi_j)$ variate. Marginalizing over $\pi_j$, we see that edges $\{A_{ij}\}_{j \neq i} \,\vert\, \pi_i$ are iid $\operatorname{Bernoulli}(\mu \pi_i )$ variates, and thus $d_i \,\vert\, \pi_i \sim \operatorname{Binomial}(n-1, \mu \pi_i )$. Writing $\P(d_i = k) = \E\!\big\{ \P(d_i = k \,\vert\, \pi_i ) \big \}$ and then substituting $t = \mu \pi_i $ for the resultant variable of integration yields~\eqref{eq:margProbD}.
\end{proof}
\vspace{-0.2\baselineskip}%

Proposition~\ref{thm:hierModels} mirrors Proposition~\ref{thm:condModels} in providing the first two moments and the distribution of each network degree. Here each degree is identically distributed, and thus any variability in degrees will be directly expressed through $\var(d)$. Once again, it is natural to compare the distribution of $d$ to a $\operatorname{Poisson}\!\big(\!\E(d)\big)$ random variable, by way of the dispersion $\var(d) / \E(d)$.

From~\eqref{eq:margED} and~\eqref{eq:margVar} we calculate the dispersion of each degree as
\begin{equation}
\label{eq:margDisp}
\frac{\var(d)}{\E(d)}
= (n-2) \sigma^2 + 1 - \mu^2 \quad (n \geq 2).
\end{equation}
Comparing to~\eqref{eq:dispDiff}, which shows each network degree to be under-dispersed conditional on $\underline{\pi}$,~\eqref{eq:margDisp} allows for marginal over-dispersion, depending on the moment behavior of $F(\pi)$ in $n$. If $\sigma^2$ remains order one as $n$ increases, then we see from~\eqref{eq:margDisp} the dispersion of $d$ will grow; to match the unity dispersion of a Poisson variate, we must have $\sigma^2 = \mu^2 / (n-2)$.

Comparing conditional and marginal dispersions in this manner illustrates the notion of variability both within and across networks. For example, one might be tempted to assume $d_i \approxeq \E(d_i \,\vert\, \underline{\pi})$ for every $i$---ignoring sampling variability across networks---or $\pi_i \approxeq \pi_j$ for all $i,j$---ignoring within-network degree variability. If both $F(\pi)$ and the Poisson--Binomial distribution of~\eqref{eq:PoissonBinomial} are explicitly acknowledged, the degrees are seen to be more heterogeneous than if either of these two sources of variation is ignored.

Finally,~\eqref{eq:margCov} implies the following degree correlation when $\var(d) > 0$:
\begin{equation*}
\frac{\cov(d, d')}{\var(d)} = \frac{1}{n-1} \left\{1 + \frac{2(n-5/2) \sigma^2}{ (n-2)\sigma^2 + 1 - \mu^2 } \right\} \quad (n \geq 3).
\end{equation*}
Thus all degrees decorrelate at rate $1/n$, in contrast to Proposition~\ref{thm:condModels}.

\subsection{Degrees as Binomial mixtures}\label{sec:binomMixDegs}

Proposition~\ref{thm:hierModels} establishes that when $\underline{\pi}$ is a random sample arising from law $F(\pi)$, then each network degree takes the Binomial mixture distribution specified by~\eqref{eq:margProbD}. For intuition, consider the special case of $F(\pi) = \I(\pi \geq \sqrt{p})$, with $\I(\cdot)$ the indicator function. This is equivalent to setting $\pi_i = \sqrt{p}$ for all $i$ in Proposition~\ref{thm:condModels}, and recovers the classical homogeneous random graph setting of $\operatorname{Binomial}(n-1,p)$ degrees.

Further study of~\eqref{eq:margProbD} leads to a fuller understanding of degree behavior. In particular, its Binomial kernel $g_n(t,k) = \binom{n-1}{k} t^k (1 - t)^{n-k-1}$ can be simplified to obtain large-sample approximations for $\P(d=k)$. Early results in this direction were established by \citet{hald1968mixed} for mixing densities $f(\pi)$ that are smooth on the entire unit interval. This condition cannot be met in our setting, because the multiplicative structure of our model implies a dilation of the mixing law by $1/\mu$. This means that the support of $f(\pi)$ will be mapped into $[0, \mu]$, and hence $\mu^{-1}f(\pi / \mu)$ will in general fail to be smooth at $\mu$. Thus a more careful analysis is necessary.

To quantify our understanding of $g_n(t,k)$ in~\eqref{eq:margProbD}, we first appeal to the de~Moivre--Laplace limit theorem. This establishes that when both $k$ and $n-k$ are order $n$, then $g_n(t,k)$ behaves locally like a $\operatorname{Normal}(nt, nt(1-t))$ kernel in variable $k$. This kernel concentrates in a neighborhood of $k = nt$ as $n \rightarrow \infty$, and thus acts like the $\delta$-distribution for a range of $k$.

The integral of $g_n(t,k)$ with respect to Lebesgue measure on $[0, \mu] \subseteq [0,1]$ is defined via the regularized incomplete Beta function $I_\mu (k+1,n-k)$:
\begin{equation}
\label{eq:IncBetaFcn}
\int_0^\mu {\textstyle \binom{n-1}{k} \, t^k (1 - t)^{n-k-1} \,dt } = \frac{1}{n} \, I_\mu (k+1,n-k).
\end{equation}
The concentration of $g_n(t,k)$ in $n$ implies that $I_\mu (k+1,n-k)$ tends toward a step function that transitions from $1$ to $0$ in a neighborhood of $k = n\mu$.

Now consider~\eqref{eq:margProbD} once again, and suppose that $F(\pi)$ admits a density $f(\pi)$ with respect to Lebesgue measure on $[0,1]$. Observe that, weakly,
\begin{equation*}
n\mu \int \delta(nt - k) \textstyle \, f\big( \frac{t}{\mu} \big) \, \frac{dt}{\mu} = f\big( \frac{k}{n\mu} \big) \quad (k \leq n \mu)
\end{equation*}
whenever $k \leq n \mu$, and from~\eqref{eq:IncBetaFcn} observe that for $f(\pi) = \I(0 \leq \pi < 1)$
\begin{equation}
\label{eq:IncBetaBinom}
n\mu \int \textstyle g_n(t,k) \, \I\!\left(0 \leq t < \mu\right) \, \frac{dt}{\mu} = I_{\mu}(k+1,n-k).
\end{equation}

Later, in Theorem~\ref{thm:repro} and what follows, we will see from a Taylor series argument that an ``approximate sifting'' property holds, such that when $f(\pi)$ is sufficiently smooth and $n \gg 1$, the heuristic relation
\begin{equation*}
n\mu \int g_n(t,k) \textstyle \, f\big( \frac{t}{\mu} \big) \, \frac{dt}{\mu} \approxeq \textstyle f\big( \frac{k}{n\mu} \wedge 1 \big) \, I_{\mu}(k+1,n-k)
\end{equation*}
holds over the entire range of $k$. This is significant because we are able to quantify precisely the behavior of $I_\mu (k+1,n-k)$ as follows.

\begin{lemma}[Binomial survival function]\label{lem:BinomSurvival}
Let $I_\mu(k+1,n-k)$ denote the regularized incomplete Beta function, with $\mu \in (0,1)$ and $(k,n): k < n$ nonnegative integers. Then $1 - I_{\mu}(k+1,n-k)$ is the law of a $\operatorname{Binomial}(n,\mu)$ variate. This implies the following, with $\Phi(\cdot)$ the law of a standard Normal:
\begin{align}
\label{eq:Hoeffding1}
1 - {\textstyle \frac{1}{2}} e^{-2(k-n\mu)^2/n} & \leq I_\mu(k+1,n-k) \leq 1 \hskip1.6cm (0 \leq k \leq n\mu),\!\!\!\!
\\ \label{eq:Hoeffding2}
0 \leq I_\mu(k+1,n-k) & \leq {\textstyle \frac{1}{2}} e^{-2(k-n\mu+1)^2/n} \hskip2cm (n\mu < k \leq n - 1);\!\!\!\!
\\ \label{eq:BinomialNormal} I_\mu(k+1,n-k) & = \textstyle 1 - \Phi\Big( \frac{k-n\mu}{\sqrt{\smash[b]{n\mu(1-\mu)}}} \Big) + \mathcal{O}\big(\frac{1}{\sqrt{\smash[b]{n}}}\big) \quad (0 \leq k \leq n-1).\!\!\!\!
\end{align}
\end{lemma}

\begin{proof}
That $I_\mu(k+1,n-k)$ is a $\operatorname{Binomial}(n,\mu)$ survival function is easily verified by applying integration by parts to~\eqref{eq:IncBetaFcn}. The exponential tail bounds of~\eqref{eq:Hoeffding1} and~\eqref{eq:Hoeffding2} are a direct consequence of Hoeffding's (\citeyear{hoeffding1963probability}) inequality for Binomial variates, and the Normal approximation of~\eqref{eq:BinomialNormal} is implied by the Berry--Esseen inequality for sums of iid random variables.
\end{proof}

Together the results of Lemma~\ref{lem:BinomSurvival} fully characterize the behavior of $I_\mu(k+1,n-k)$ when $n$ is large. We see from~\eqref{eq:Hoeffding1} and~\eqref{eq:Hoeffding2} respectively that $I_{\mu}(k+1,n-k)$ goes exponentially quickly to $1$ or to $0$ whenever $k-n\mu = \omega(\sqrt{n})$; i.e., whenever $k$ is chosen such that $|k-n\mu| / \sqrt{n}$ diverges in $n$. In the region $k-n\mu = \mathcal{O}(\sqrt{n})$, the function $I_{\mu}(k+1,n-k)$ transitions from $1$ to $0$, with~\eqref{eq:BinomialNormal} relating it to a standard Normal distribution function. This overall behavior is illustrated in Fig.~\ref{fig:IncBetaFcn}, which shows $I_{1/2}(k+1,500-k)$, as well as the surface describing $I_{\mu}(k+1,n-k)$ as a function of both $\mu$ and $k$.

\begin{figure}[t]
\begin{center}
\includegraphics[width=1.05\columnwidth]{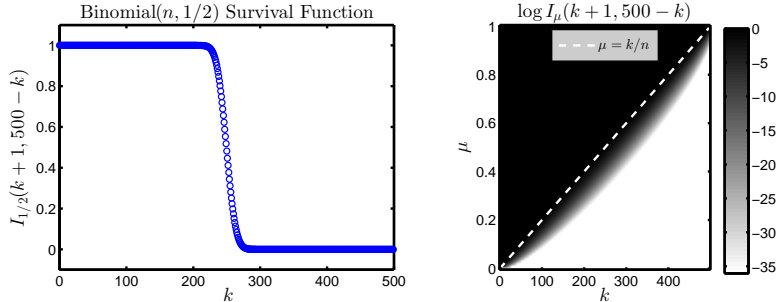}
\caption{\label{fig:IncBetaFcn} $\!\!$Large-sample behavior of the Binomial survival function $I_{\mu}(k\!+\!1,n\!-\!k)$ as a function of $k$ and $(k,\mu)$. The dashed line at right shows the midpoint $\mu = k/n$ of its transition region.$\!$}
\end{center}
\vspace{-0.75\baselineskip}%
\end{figure}

More generally, we see in analogy to~\eqref{eq:IncBetaBinom} that
\begin{align}
\label{eq:incBetaDiff}
n\mu \int_{\mu a}^{\mu b} \textstyle g_n(t,k) \, \frac{dt}{\mu} 
& = I_{\mu b} (k+1,n-k)- I_{\mu a} (k+1,n-k)
\\ \nonumber & \equiv
\left( I_{\mu b} - I_{\mu a} \right) (k+1,n-k).
\end{align}
From these $\operatorname{Binomial}(n,\mu b)$ and $\operatorname{Binomial}(n,\mu a)$ survival functions, we will observe an effect similar to Fig.~\ref{fig:IncBetaFcn}, in which censoring occurs both below $n \mu a $ and above $n \mu b $. This follows from Lemma~\ref{lem:BinomSurvival}, as when both of these functions are near to $0$ or $1$, their difference is effectively zero.

We next proceed to examine $F(\pi)$, the second component of~\eqref{eq:margProbD}. This will enable us to simplify the expressions of Proposition~\ref{thm:condModels} for conditional means and covariances, as well as their marginal counterparts in Proposition~\ref{thm:hierModels}. It is natural to begin with decaying deterministic sequences $\underline{\pi}$, and we will start with a special choice of decay that leads to power law degrees.

\section{Power law networks}\label{sec:power-law}

We saw earlier in Section~\ref{sec:multModel} that conditional moments of network degrees depend on the relative and absolute magnitudes of $\| \underline{\pi} \|_1$ and $\| \underline{\pi} \|_2^2$, which in turn may grow with $n$. If we order elements of $\underline{\pi}$ from largest to smallest, so that $\pi_1 \geq \pi_2 \geq \cdots \geq \pi_n$, then it is natural to model their decay directly. Indeed, a model for the decay of $\pi_i$ in $i$ will determine how each expected degree $\E(d_i \,\vert\, \underline{\pi})$ grows with $n$. If elements of $\underline{\pi}$ decay slowly, then $\| \underline{\pi} \|_1$ and $\| \underline{\pi} \|_2^2$ will grow at the same rate in $n$; if elements of $\underline{\pi}$ decay too quickly in $i$, then neither $\| \underline{\pi} \|_1$ nor $\| \underline{\pi} \|_2^2$ will exhibit growth with $n$; but if elements of $\underline{\pi}$ exhibit controlled variability, then $\| \underline{\pi} \|_1$ can grow faster than $\| \underline{\pi} \|_2^2$, squeezing the dispersion of $d_i$ to unity.

\vspace{-0.33\baselineskip}%
\subsection{Modeling decay of expected network degrees}
\label{sec:mod_decay}

To treat interesting regimes of behavior in realized networks, we assume a polynomial decay of $\pi_i$ with $i$. Following \cite{chung2003eigenvalues}, we fix $\gamma \in (0,1)$ and take $\pi_i \propto i^{-\gamma}$. This enables us to control network degree variability parametrically, and leads to networks that have power laws as limiting degree distributions. These have seen significant study in the applied literature \citep{Durrett}. A basic characterization is the expected proportion of degree-$k$ nodes in a single $n$-node network realization, which scales as $k^{-(1+1/\gamma)}$ when all network degrees decorrelate in $n$. We may further characterize degree sequences under such decay models as follows.

\vspace{-0.33\baselineskip}%
\begin{theorem}[Power law degrees]\label{thm:powerLaws}
Fix an exponent $\gamma \in (0,1)$ and a sequence $\{\theta_n\}$ of scaling constants, each taking values in $[0,1]$. Let
\begin{equation}
\label{eqn:power-law}
\pi_i = \theta_n i^{-\gamma} \quad (1\le i \le n),
\end{equation}
and consider a simple graph with independent $\operatorname{Bernoulli}(\pi_i \pi_j)$ edges. Then:
\begin{enumerate}

\item As a function of $n$, the expected value of the $i$th network degree $d_i$ is
\begin{equation}
\label{eq:powLawExpDeg}
\E(d_i \,\vert\, \underline{\pi}) = \frac{\theta_n^2}{1-\gamma} i^{-\gamma} n^{1-\gamma} \big\{ 1 + \mathcal{O}(n^{-(1-\gamma)}) \big\} \quad (1 \leq i \leq n).
\end{equation}

\item As $n \rightarrow \infty$, the law of each $d_i$ converges in total variation to a $\operatorname{Poisson}\!\big(\!\E(d_i \,\vert\, \underline{\pi})\big)$ distribution; i.e., the sum of all absolute differences
\begin{equation}
\label{eq:TVnorm}
\big| \P(d_i = k \,\vert\, \underline{\pi}) - e^{-\E(d_i \,\vert\, \underline{\pi})} \E(d_i \,\vert\, \underline{\pi})^{k} / k! \,\big| \quad (0 \leq k \leq n-1),
\end{equation}
converges to zero, and thus so does each individual difference in turn.

\item Whenever $\E(d_i \,\vert\, \underline{\pi})$ grows in $n$ as determined by~\eqref{eq:powLawExpDeg}, then
\begin{equation}
\label{eq:DegConv}
\frac{d_i}{\E(d_i \,\vert\, \underline{\pi})} \overset{P}{\longrightarrow} 1
\quad \text{and} \quad \frac{d_i - \E(d_i \,\vert\, \underline{\pi})}{\sqrt{\E(d_i \,\vert\, \underline{\pi})}} \overset{L}{\longrightarrow} \operatorname{Normal}(0,1).
\end{equation}

\end{enumerate}
\end{theorem}

\vspace{-0.75\baselineskip}%
\begin{proof}
The main component driving the theorem is the relative growth of $\| \underline{\pi} \|_1$ and $\| \underline{\pi} \|_2^2$ in $n$. The growth rates of these terms follow from~\eqref{eqn:power-law} by a standard integral squeezing argument: For fixed $\delta$ and increasing $n$,
\begin{equation}
\label{eq:intSqz}
\sum_{i=1}^n i^{-\delta} =
\begin{cases}
(1-\delta)^{-1}n^{1-\delta} + \mathcal{O}(1) & \text{if $0 < \delta < 1$,} \\
\log n + \gamma_{\mathrm{E}} + \mathcal{O}(n^{-1}) & \text{if $\delta = 1$,} \\
\zeta(\delta) + \mathcal{O}(n^{-(\delta-1)}) & \text{if $\delta > 1$;}
\end{cases}
\end{equation}
with $\gamma_{\mathrm{E}}$ the Euler--Mascheroni constant and $\zeta(\cdot)$ the Riemann zeta function.
\begin{enumerate}
\item Since $\E(d_i \,\vert\, \underline{\pi}) = \pi_i \| \underline{\pi} \|_1 - \pi_i^2$, we obtain~\eqref{eq:powLawExpDeg} by setting $\delta = \gamma$ in~\eqref{eq:intSqz}.

\item \citet{barbour1992poisson} show that the total variation distance between the laws of $d_i$ and a $\operatorname{Poisson}\!\big(\!\E(d_i \,\vert\, \underline{\pi})\big)$ variate is of order
    \begin{equation}
    \label{eq:TVorder}
    \min\left\{ \E(d_i \, \vert \, \underline{\pi}),1 \right\} \left\{ 1-\frac{\var(d_i \, \vert \, \underline{\pi})}{\E(d_i \, \vert \, \underline{\pi})} \right\}.
    \end{equation}
    Applying~\eqref{eq:intSqz} to~\eqref{eq:disp} shows that $\lim_{n\rightarrow\infty} \var(d_i \,\vert\, \underline{\pi}) / \E(d_i \,\vert\, \underline{\pi}) = 1$ for all $\gamma \in (0,1)$. Thus~\eqref{eq:TVorder} is squeezed to $0$ as $n \rightarrow \infty$.
\item Growth of $\E(d_i \,\vert\, \underline{\pi})$ in $n$ implies $\lim_{n\rightarrow\infty} \var(d_i \,\vert\, \underline{\pi}) / \E(d_i \,\vert\, \underline{\pi})^2 = 0$. Thus $d_i / \E(d_i \,\vert\, \underline{\pi}) \inlineoverset{P}{\longrightarrow} 1$ by Chebyshev's inequality. From~\eqref{eq:intSqz} we conclude that growth of $\E(d_i \,\vert\, \underline{\pi})$ implies growth of $\var(d_i \,\vert\, \underline{\pi})$---a condition sufficient for the Lindeberg--Feller central limit theorem to apply.
\end{enumerate}
\vspace{-1.42\baselineskip}%
\end{proof}

Theorem~\ref{thm:powerLaws} provides three main conclusions. First,~\eqref{eq:powLawExpDeg} shows that when $\pi_i$ decays as $i^{-\gamma}$, the expected degrees also decay as $i^{-\gamma}$, with an aggregate scaling in $n$. Second,~\eqref{eq:TVnorm} shows that each $d_i$ behaves like a Poisson variate for large $n$. Finally,~\eqref{eq:DegConv} shows that for growing degrees, the relative distance between $d_i$ and $\E(d_i \, \vert \, \underline{\pi})$ shrinks as $n$ increases, with suitably scaled deviations $d_i - \E(d_i \,\vert\, \underline{\pi})$ becoming Normal. Together these three results describe the behavior of degrees under the polynomial decay model of~\eqref{eqn:power-law}.

The main driver of this result is the relative growth of $\| \underline{\pi} \|_1$ and $\| \underline{\pi} \|_2^2$ with $n$, which in turn depends on the polynomial decay specified by~\eqref{eqn:power-law}. We may relax the precise form of~\eqref{eqn:power-law} by allowing deviations from an overall polynomial decay; as long as these deviations are controlled---something we might expect in real networks---our results will still hold.
We do this by introducing a function $\xi(x): (0,1] \rightarrow [\xi_{\min}, \xi_{\max}] \subset {\mathbb{R}}^+$, where constants $\xi_{\min}$ and $\xi_{\max}$ constrain the excursion of $\xi(\cdot)$, and redefining
\begin{equation}
\label{eqn:power-law2}
\pi_i = \xi\!\left(\textstyle\frac{i}{n}\right) \theta_n i^{-\gamma} \quad (1 \leq i \leq n) .
\end{equation}

The function $\xi(x)$ thus absorbs any redundant variation that does not alter the overall polynomial decay of $\underline{\pi}$ with $i$. The model of~\eqref{eqn:power-law2} is then semi-parametric, because it enforces a direct constraint on the decay of $\underline{\pi}$ via a single parameter $\gamma$, but allows for functional variability from it; see related work in time series by \cite{robinson1994semiparametric}. This allows the essence of Theorem~\ref{thm:powerLaws} to be retained, in that each $\E(d_i\,\vert\,\underline{\pi})$ can remain of order $i^{-\gamma} n^{1-\gamma}$ for $n$ large. This permits us to deduce Poisson convergence in total variation norm, convergence of $d_i/\E(d_i\, \vert \, \underline{\pi})$ to $1$ in probability, and convergence in law of a suitably rescaled version of $d_i - \E(d_i \,\vert\, \underline{\pi})$ to a standard Normal.

\subsection{Non-parametric inference for power-law networks}

We now turn our attention to non-parametric estimation of $\underline{\pi}$ from an observed network. Theorem~\ref{thm:powerLaws} showed that whenever expected degree $\E(d_i \,\vert\, \underline{\pi})$ grows in $n$, the limiting Poisson distribution of $d_i$ becomes Normal. Moreover, when $n$ is large and the moment estimator $\hat \pi_i = d_i / \sqrt{\|\underline{d}\|_1}$ is used to estimate $\pi_i$, then $\hat \pi_i$ is approximately distributed as a $\operatorname{Normal}(\pi_i, \pi_i/\| \underline{\pi} \|_1)$ random variable.

\begin{theorem}[Power law central limit theorem]\label{thm:CLT}
Consider a sequence of $n$-node simple graphs with independent $\operatorname{Bernoulli}(\pi_i \pi_j)$ edges. If $\underline{\pi}$ follows the power law model of Theorem~\ref{thm:powerLaws}, then whenever $\E(d_i \,\vert\, \underline{\pi})$ grows in $n$,
\begin{equation}
\label{eqn-non-parest}
\frac{d_i / \sqrt{\|\underline{d}\|_1} - \pi_i}{ \sqrt{ \pi_i / \| \underline{\pi} \|_1}} \overset{L}{\longrightarrow} \operatorname{Normal}(0,1).
\end{equation}
\end{theorem}

\begin{proof}
By Theorem~\ref{thm:powerLaws}, $\{d_i-\E( d_i\,\vert\, \underline{\pi})\} / \sqrt{\E(d_i\,\vert\, \underline{\pi})} \inlineoverset{L}{\longrightarrow} \operatorname{Normal}(0,1)$ whenever $\E(d_i \,\vert\, \underline{\pi}) \rightarrow \infty$. We obtain a limit theorem for $\hat{\pi}_i = d_i / \sqrt{\|\underline{d}\|_1}$ by writing
\begin{equation*}
\frac{\hat{\pi}_i-\pi_i}{\sqrt{ \pi_i / \| \underline{\pi} \|_1}}
= 
\Big\{
\underbrace{ \textstyle \sqrt{\frac{\pi_i\| \underline{\pi} \|_1}{\E(d_i\,\vert\, \underline{\pi})} \cdot \frac{ \|\underline{d}\|_1}{\| \underline{\pi} \|_1^2}}
 }_{\text{I}}
\Big\}^{-1}
\bigg\{ \frac{d_i-\E( d_i\,\vert\, \underline{\pi})}{\sqrt{\E(
d_i\,\vert\, \underline{\pi})}}
+ \underbrace{ \textstyle \frac{\E (d_i\,\vert\, \underline{\pi})-\pi_i\sqrt{\|\underline{d}\|_1}}{\sqrt{\E(
d_i\,\vert\, \underline{\pi})}} }_{\text{II}}
\bigg\},
\vspace{-0.33\baselineskip}%
\end{equation*}
and showing that terms~I and~II converge in probability to $1$ and $0$, respectively. Applying Slutsky's theorem then establishes the result as claimed.

Term I is straightforward: Since $\E(d_i \,\vert\, \underline{\pi}) = \pi_i (\| \underline{\pi} \|_1 - \pi_i)$, its growth in $n$ implies $\sqrt{\pi_i \| \underline{\pi} \|_1 / \E(d_i\,\vert\, \underline{\pi}) } \rightarrow 1$, and since $\| \underline{d} \|_1 / 2$ is the sum of all edges, we calculate $\E( \| \underline{d} \|_1 \,\vert\, \underline{\pi}) = 2\sum_{i<j} \pi_i\pi_j = \| \underline{\pi} \|_1^2 - \| \underline{\pi} \|_2^2$ and $\var( \| \underline{d} \|_1 \,\vert\, \underline{\pi}) = 4 \sum_{\smash{i<j}} \pi_i\pi_j(1-\pi_i\pi_j) $. By Chebyshev's inequality, these moments imply $\| \underline{d} \|_1 / \|\underline{\pi}\|_1^2 \inlineoverset{P}{\longrightarrow} 1$, so by the continuous mapping theorem, $\sqrt{\|\underline{d}\|_1} / \|\underline{\pi}\|_1 \inlineoverset{P}{\longrightarrow} 1$.

Term II is more delicate, as its numerator must balance. Write term II as
\begin{equation*}
\frac{ \E\left(d_i\,\vert\, \underline{\pi}\right) \!-\! \pi_i \sqrt{\|\underline{d}\|_1}}{\sqrt{\E\left(d_i\,\vert\,
\underline{\pi}\right)}}
= \underbrace{ \textstyle \frac{ \E\left(d_i\,\vert\, \underline{\pi}\right) \!-\! \pi_i \sqrt{\E\left(\|\underline{d}\|_1
\,\vert\, \underline{\pi}\right)}}{\sqrt{\E\left(d_i\,\vert\, \underline{\pi}\right)}}
 }_{\text{IIA}}
- \underbrace{ \textstyle \frac{\pi_i \myl\{ \sqrt{\|\underline{d}\|_1} \!-\! \sqrt{\E\left(\|\underline{d}\|_1 \,\vert\,
\underline{\pi}\right)} \,\myr\}}{\sqrt{\E\left(d_i\,\vert\, \underline{\pi}\right)}}
 }_{\text{IIB}},
\vspace{-0.33\baselineskip}%
\end{equation*}
and observe that term IIA converges to zero if $\E(d_i \,\vert\, \underline{\pi}) \rightarrow \infty$. To show that term IIB, and hence term II, converges in probability to $0$, it is sufficient to prove that $\sqrt{\|\underline{d}\|_1} - \sqrt{\E\left(\|\underline{d}\|_1 \,\vert\, \underline{\pi}\right)}$ is bounded in probability when $\E(d_i \,\vert\, \underline{\pi}) \rightarrow \infty$. Lemma~\ref{lem:deltaMethod} shows this via a Taylor expansion of $\sqrt{\|\underline{d}\|_1 / \E(\|\underline{d}\|_1 \,\vert\, \underline{\pi})}$.
\end{proof}

The growth of $\E(d_i \,\vert\, \underline{\pi})$ in $n$ drives the asymptotic Normality of $\hat{\pi}_i$ in Theorem~\ref{thm:CLT}, and gives rise to large-sample interval estimates and hypothesis tests for $\hat{\pi}_i$ when the model $\pi_i = \theta_n i^{-\gamma}$ is in force. Furthermore, observe from~\eqref{eq:powLawExpDeg} that if we constrain $\gamma \in (0, 1/2)$ and $\theta_n = \omega( 1/n^{1/2-\gamma} )$, then all degrees will grow in $n$, and consequently Theorem~\ref{thm:CLT} will apply to each element of the vector-valued estimator $\underline{\hat \pi} = \underline{d} / \sqrt{\|\underline{d}\|_1} $. In this case the Cram\'er--Wold device will enable a full multivariate understanding.

These methods of analysis will also apply outside of the power-law setting of Theorem~\ref{thm:CLT}, in other cases where $\E(d_i \,\vert\, \underline{\pi})$ grows in $n$. One possibility, discussed in Section~\ref{sec:mod_decay}, is to assume a polynomial decay in $i$ that varies according to some envelope function $\xi( i / n )$ in accordance with~\eqref{eqn:power-law2}. In this case the estimator $\underline{\hat \pi} = \underline{d} / \sqrt{\|\underline{d}\|_1} $ enables an initial non-parametric exploratory analysis of the data, which can then be refined when a suitable parametric form has been identified. Theorem~\ref{thm:CLT} thus provides a first step toward a more general understanding of goodness-of-fit for network models---an important open challenge recently highlighted by \citet{fienberg2012brief}.

\subsection{A population of power-law networks}

To allow for more variable network degree realizations, we may wish to model $\underline{\pi}$ as a random sample from law $F(\pi)$. Motivated by the polynomial decay model $\pi_i \propto i^{-\gamma}$ of Section~\ref{sec:mod_decay}, we consider the case of a Pareto distribution restricted to a subset $[a,b]$ of the unit interval. This amounts to taking $F(\pi) \propto \pi^{1-\beta} \I(a \leq \pi < b)$ for some $\beta \geq 0$, and we can further the correspondence with Section~\ref{sec:mod_decay} by setting $\beta = 1 + 1 / \gamma$. This is because the ordered elements $\pi_{(1)} \geq \pi_{(2)} \geq \cdots \geq \pi_{(n)}$ of a random sample from any $F(\pi)$ will obey $\E\left(\pi_{(i)}\right) \rightarrow F^{\smash{-1}} \left(i/n\right)$ as $n \rightarrow \infty$, and here $F^{\smash{-1}} \left(i/n\right)\propto i^{\smash{1/(1-\beta)}} = i^{-\gamma}$. We now show that the degrees take a particularly simple form when $\underline{\pi}$ is sampled from a bounded Pareto density.

\begin{theorem}[Power law populations]\label{thm:powerLawRepro}
Let $\pi_1, \ldots \pi_n$ be a random sample from a bounded Pareto density $f(\pi) \propto \pi^{-\beta} \I(a \leq \pi < b)$, where $\beta \geq 0$, and $a > 0$ whenever $\beta \geq 1$ to ensure integrability. Then for $\beta < k \leq n-1$,
\begin{equation}
\label{eq:power-lawrep}
n\mu \, \P(d = k) = \textstyle c \; \big( \frac{k}{n \mu} \big)^{-\beta} \left( I_{\mu b} - I_{\mu a} \right) (k+1-\beta,n-k) \, \big\{ 1 + \epsilon_{k,n}(\beta) \big\},
\end{equation}
where $\epsilon_{k,n}(\beta) = \beta(\beta-1) \!\left\{ (n-k)/(2nk)+ \mathcal{O}\!\left( k^{-2} \right) \right\}$, and $c^{-1} = \int_a^b \pi^{-\beta} \, d\pi$ is the normalizing constant of $f(\pi)$. The error term $\epsilon_{k,n}(\beta)$ vanishes if $\beta =0$ or $\beta = 1$, corresponding to the cases in which $f(\pi)$ is uniform or linear.
\end{theorem}

\begin{proof}
Substituting $ \mu^{-1} c \; (t/\mu)^{-\beta} \,dt$ for $dF(t/\mu)$ in~\eqref{eq:margProbD}, we see that
\begin{equation*}
\P(d = k) = \frac{\binom{n-1}{k} \, c}{\mu^{1-\beta}} \int_{\mu a}^{\mu b} t^{k-\beta} (1 - t)^{n-1-k} \,dt \quad (0 \leq k \leq n-1).
\end{equation*}
Multiplying by $1 \!=\! \{ \Gamma(k + 1 - \beta) \Gamma(n-k) / \Gamma(n + 1 - \beta) \} / \! \int_0^1 t^{k\!-\!\beta}(1\!-t)^{n\!-\!1\!-\!k} dt $, the integral being defined for $k > \beta - 1$, and combining with $\binom{n-1}{k}$, we obtain
\begin{equation*}
\P(d = k) = \frac{c}{n\mu^{1-\beta}} \frac{\Gamma(n+1)}{\Gamma(n+1-\beta)} \frac{\Gamma(k+1-\beta)}{\Gamma(k+1)}
\,\, \frac{\int_{\mu a}^{\mu b} t^{k-\beta}(1-t)^{n-1-k} \, dt}{\int_0^1 t^{k-\beta}(1-t)^{n-1-k} \, dt}
\end{equation*}
for $\beta - 1 < k \leq n-1$. Recognizing $\left( I_{\mu b} - I_{\mu a} \right) (k+1-\beta,n-k)$ from~\eqref{eq:incBetaDiff},
\begin{equation*}
\frac{\Gamma(n\!+\!1)}{\Gamma(n\!+\!1\!-\!\beta)} \frac{\Gamma(k\!+\!1\!-\!\beta)}{\Gamma(k\!+\!1)}
= \left( \frac{k\!+\!1}{n\!+\!1} \right)^{\!-\beta} \!\!
 \Big\{ 1 + \textstyle \frac{\beta(\beta+1)}{2} \! \left(\frac{n-k}{(n+1)(k+1)} \right) + \mathcal{O}\!\left( \frac{1}{(k+1)^2} \right) \! \Big\}
\end{equation*}
follows by Lemma~\ref{thm:gammaRatio}. Taylor-expanding $k+1$ and $n+1$ then yields~\eqref{eq:power-lawrep}.
\end{proof}

Theorem~\ref{thm:powerLawRepro} shows that when $n$ is large and $f(\pi) \propto \pi^{-\beta} \I(a \leq \pi < b)$, network degrees will reflect this same power-law distribution, in that $\P(d = k)$ will scale with $k$ in approximate proportion to $k^{-\beta}$ over the range $n \mu a < k < n \mu b$. Indeed, recalling the discussion of Section~\ref{sec:binomMixDegs}, the difference $\left( I_{\mu b} - I_{\mu a} \right) (k+1-\beta,n-k)$ of regularized incomplete Beta functions will restrict the set of likely values of any degree $d$ to within this range.

Well inside this range, the difference $\big( I_{\mu b} - I_{\mu a} \big) (k+1-\beta,n-k)$ will be nearly unity, in accordance with Lemma~\ref{lem:BinomSurvival}, and the power-law effect will be visible. In the transition regions $|k - n \mu a| = \mathcal{O}(\sqrt{n})$ and $|k - n \mu b| = \mathcal{O}(\sqrt{n})$, Lemma~\ref{lem:BinomSurvival} verifies that $\big( I_{\mu b} - I_{\mu a} \big) (k+1-\beta,n-k)$ will behave like a Normal distribution function. Outside of these regions, the tail decay of either $I_{\mu a}(k+1-\beta,n-k)$ or $I_{\mu b}(k+1-\beta,n-k)$ will rapidly dominate, and so the probability of observing $d=k$ will decay exponentially in $n$. As shown in Fig.~\ref{fig:powLawsFitted}, this corresponds to a censoring of the power law effect.

\begin{figure}[t]
\begin{center}
\includegraphics[width=0.94\columnwidth]{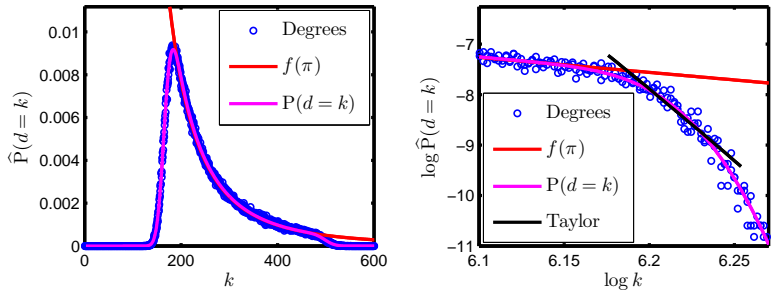}\vspace{-0.7\baselineskip}%
\caption{\label{fig:powLawsFitted} Left panel: Averaged degrees from power-law networks with $n=1000$ nodes, generated from $f(\pi) \propto \pi^{-3} \I (1/3 \leq \pi < 1)$, with $\mu = 1/2$. Rapid decays are visible near $k = n \mu a = 500/3$ and $k = n \mu b = 500$. Right panel: a log-log plot of the transition region near $\log ( n \mu b ) \approxeq 6.2$, showing the exact and empirical distributions along with a Taylor expansion of $I_{\mu b}(k+1-\beta,n-k)$, corresponding to an exponential cutoff effect.\vspace{-0.27\baselineskip}}
\end{center}
\vspace{-0.75\baselineskip}%
\end{figure}

\subsection{Censoring of extreme degrees}

To investigate this censoring effect, the left-hand panel of Fig.~\ref{fig:powLawsFitted} shows empirical frequencies $\widehat \P(d = k)$ of degrees generated from $1000$-node networks in which $f(\pi) \propto \pi^{-3} \I (1/3 \leq \pi < 1)$, with $\P(d=k)$ given by~\eqref{eq:power-lawrep}. As predicted, we observe a rapid exponential censoring effect outside of the lower and upper transition regions near $k = n \mu a$ and $k = n \mu b$, and a power law decay matched to $f(\pi)$ otherwise.

The right-hand panel of Fig.~\ref{fig:powLawsFitted} illustrates the upper transition region on a logarithmic scale, along with a Taylor expansion of $I_{\mu b}(k+1-\beta,n-k)$, which matches that observed censoring effect to first order. This ``exponential cutoff'' of extreme degrees is frequently observed in practice, and has motivated models that explicitly parameterize its effects \citep{newman2001structure}.

Our analysis shows that this observed exponential cutoff can be a natural consequence of a fully generative model, rather than a property of any given data set. In practice, power laws are typically identified from data by taking logarithms of empirical frequencies $\widehat \P(d = k)$, and then inferring a linear trend. Theorem~\ref{thm:powerLawRepro} shows near the largest observed degrees that $\log \, \P(d = k)$ is a sum of contributions from $\log f\big(k/(n\mu)\big)$ and $\log I_{\mu b}(k+1-\beta,n-k)$, and thus effects ascribed to a model of degree sequence behavior may in fact be due purely to the effects of sampling. Such effects, derived using a valid statistical model for the observations, must be included when evaluating the properties of any subsequent network estimators.

\vspace{-0.08\baselineskip}%
\section{Network populations parameterized by smooth distributions}
\label{sec:net-pop}

As discussed in Section~\ref{sec:binomMixDegs}, results similar to Theorem~\ref{thm:powerLawRepro} hold more generally, for densities $f(\pi)$ that are smooth enough to admit a bounded second derivative. In this case we are able to directly characterize the interaction between the Binomial kernel $g_n(\cdot,k)$ and $f(\pi)$ itself. When $f(\pi)$ is fixed, the resulting networks will be dense; their expected degrees will scale linearly in $n$. The variability resulting from $f(\pi)$ will eventually swamp the sampling variation inherent to each Bernoulli edge, so that $f(\pi)$ essentially determines the distribution of each network degree. The following result quantifies this behavior in terms of the key quantities $n \mu$, $f(\cdot)$, and $I_{\mu}(k\!+\!1, n\!-\!k)$. We will later show that it also extends to sparse network regimes.

\vspace{-0.08\baselineskip}%
\begin{theorem}\label{thm:repro}
Suppose $f(\pi)$ is continuous and nonzero on $[0,1]$, and twice differentiable on $(0,1)$ with bounded second derivative. Then
\begin{equation*}
\textstyle n \mu \P(d = k) = f\!\left( \frac{(k+1)\,\iota_{k,n}(\mu)}{(n+1)\mu} \right) I_{\mu}(k\!+\!1, n\!-\!k) \, \left\{ 1 + \mathcal{O}\!\left( \frac{1}{n \mu} \right) \right\},
\end{equation*}
where $\iota_{k,n}(\mu) \in [\mu, 1)$ is defined via a $\operatorname{Binomial}(n,\mu)$ random variable $X_n$ as
\begin{equation}
\label{eq:rnk}
\iota_{k,n}(\mu) = 1 - (1-\mu) \, \frac{ \P(X_n = k+1) }{ \P(X_n \geq k+1)} \quad (0 \leq k \leq n - 1).
\end{equation}
The argument of $f$ in~\eqref{eq:rnk} is strictly concave and increasing from $0$ to $1$, approaching $\mu^{-1}(k\!+\!1)/(n\!+\!1)$ when $k \ll n \mu$ and $(k\!+\!1)/(k\!+\!2)$ when $k \gg n \mu$.

Even if $f$ attains zero on $[0,1]$, we have for $c = \sup_{\pi \in (0,1)} \left| f''(\pi) \right| / 2$ that
\begin{equation*}
\textstyle
 \left| n \mu \P(d = k) - f\!\left( \frac{(k+1)\,\iota_{k,n}(\mu)}{(n+1)\mu} \right) I_{\mu}(k\!+\!1, n\!-\!k) \right|
< \frac{c}{n\mu} \, \frac{(k+1)\,\iota_{k,n}(\mu)}{(n+1)\mu} \, I_{\mu}(k+1, n-k).
\end{equation*}
\end{theorem}

\begin{proof}
A Taylor expansion of $f(\cdot)$ at a carefully chosen point yields the result. Recall from Proposition~\ref{thm:hierModels} that each degree $d$ is a Binomial mixture with mixing distribution $F(t / \mu)$. Since $dF(t / \mu) = \mu^{-1} f(t / \mu) \, dt$ by hypothesis, multiplying both sides of~\eqref{eq:margProbD} by $n \mu / I_{\mu}(k+1,n-k)$ yields
\begin{equation}
\label{eq:BetaRepro}
\frac{n \mu \P(d = k)}{I_{\mu}(k+1,n-k)} = \frac{n \, \binom{n-1}{k}}{I_{\mu}(k+1,n-k)} \int_0^{\mu} \! t^k (1 - t)^{n-1-k} \, \textstyle f\big( \frac{t}{\mu} \big) \,dt.
\end{equation}
We recognize the right-hand side as the expected value of $f(T / \mu)$, where $T$ is a truncated $\operatorname{Beta}(k+1,n-k)$ random variable with moments
\begin{align*}
\E(T) & = \left( \frac{k+1}{n+1} \right) \frac{ I_{\mu}(k+2,n-k) }{I_{\mu}(k+1,n-k)},
\\ \var(T)
& = \E(T) \left\{ \left( \frac{k+2}{n+2} \right) \frac{I_{\mu}(k+3,n-k)}{I_{\mu}(k+2,n-k)} - \E(T) \right\}
< \frac{\E(T)}{n+2},
\end{align*}
with Lemma~\ref{thm:VarBnd} establishing the inequality. Under the theorem hypothesis, we may write $f(t / \mu)$ using Lagrange's form of Taylor's remainder: \begin{equation*}
\textstyle f\!\left( \frac{t}{\mu} \right) = f\!\left( \frac{\E(T)}{\mu} \right) + f'\!\left( \frac{\E(T)}{\mu} \right) \left( \frac{t-\E(T)}{\mu} \right) + \frac{1}{2} f''\!\left( \frac{\tau(t)}{\mu} \right) \Big( \frac{t-\E(T)}{\mu} \Big)^2,
\end{equation*}
where $t < \tau(t) < \E(T)$. Substituting this into~\eqref{eq:BetaRepro}, the mean term will vanish, and the variance term is bounded above by $\E(T) / (n+2)$. Hence
\begin{equation*}
\left| \frac{ n \mu \P(d = k)}{I_{\mu}(k+1,n-k)} - {\textstyle f\!\left( \frac{\E(T)}{\mu} \right) } \right|
< \frac{1}{2\mu^2} \sup_{\tau \in (0,\mu)} \left| {\textstyle f''\!\left( \frac{\tau}{\mu} \right) } \right| \, \frac{\E(T)}{n+2}.
\end{equation*}
Applying integration by parts to $I_\mu(k+2,n-k)$, we recover the identity of~\eqref{eq:rnk} by setting $\iota_{k,n}(\mu) = I_\mu(k+2,n-k) / I_\mu(k+1,n-k)$. Implications for the argument of $f$ follow from the concentration inequalities of Lemma~\ref{lem:BinomSurvival}, and from Taylor expansions of $\iota_{k,n}(\mu)$ near $k = 0$ and $k = n-1$.
\end{proof}

In essence, Theorem~\ref{thm:repro} states that when $f(\cdot)$ is smooth and $n \gg 1$,
\begin{equation}
\label{eq:approxRepro}
\P(d = k) \approxeq
\textstyle \frac{1}{n \mu} f\big( \frac{k}{n \mu} \wedge 1 \big)
\, \Big\{ 1 - \Phi\Big( \frac{k - n \mu }{\sqrt{ \smash[b]{n \mu (1-\mu)}}} \Big) \Big\}.
\end{equation}

\begin{figure}[t]
\begin{center}
\includegraphics[width=0.825\columnwidth]{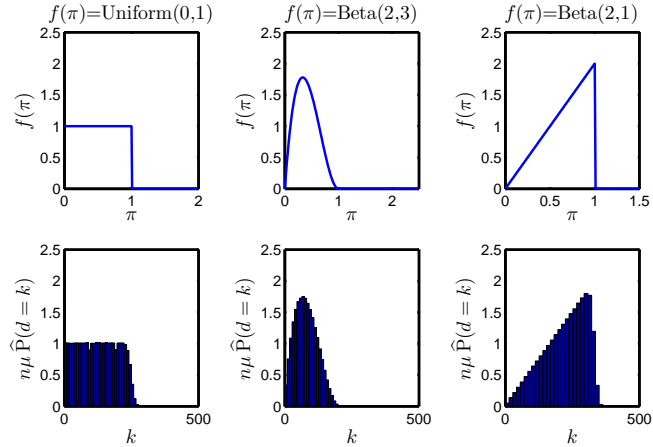}\vspace{-0.7\baselineskip}%
\caption{\label{fig:reproIllus} Scaled empirical degree distributions $n \mu \, \widehat \P(d = k)$ illustrating Theorem~\ref{thm:repro}, with $f(\pi)$ (top row) reproduced on the interval $0 \leq k \leq n\mu$ (bottom row; $n=500$). Note the right tail behavior apparent in the bottom row, which follows from the relation of~\eqref{eq:approxRepro}.}
\end{center}
\vspace{-0.75\baselineskip}%
\end{figure}

Figure~\ref{fig:reproIllus} illustrates this result through network simulations for three choices of smooth $f(\cdot)$. It reveals both the primary effect of $f(\cdot)$ being reproduced via $\P(d = k)$ for $k \leq n \mu$, as well as the finer effects of the Binomial survival function $I_{\mu}(k+1,n-k)$ for $k > n \mu$, due to large but finite $n$.

The interpretation of~\eqref{eq:approxRepro} comes by way of the Normal approximation to $I_{\mu}(k\!+\!1, n\!-\!k)$ given by Lemma~\ref{lem:BinomSurvival}, and Taylor expansions of the argument of $f(\cdot)$ about $k/(n\mu)$ and $1$. One way to formalize this analysis is to split the integral of~\eqref{eq:BetaRepro} into multiple $n$-dependent regions, in order to take advantage of the concentration inequalities of Lemma~\ref{lem:BinomSurvival} directly.

More generally, the results of Theorem~\ref{thm:repro} improve substantially on those of \citet{hald1968mixed} for a more restrictive class of mixed Binomial distributions. They reveal the essence of the multiplicative sampling underlying our degree-based network model. Recalling the discussion of Section~\ref{sec:binomMixDegs}, this becomes particularly important when $|k - n \mu| = \mathcal{O}( \sqrt{n} )$, as the sampling mechanism itself imposes an increasingly strong truncation effect on $\P(d = k)$. Beyond this range, $\P(d = k)$ decays exponentially, in accordance with Lemma~\ref{lem:BinomSurvival}.

\section{Scaling regimes}
\label{sec:scaling}

Real networks can exhibit substantial heterogeneity in their observed degree distributions. This observed heterogeneity has already been partially explored in Theorems~\ref{thm:powerLaws} and~\ref{thm:repro}. These theorems explain different aspects of heterogeneity introduced both by the sampling variability of a fixed degree $d_i$, conditional upon $\{\pi_1, \ldots \pi_n\}$, and the marginal variability of any degree $d$ generated from $\{\pi_1, \ldots \pi_n\}$ via some smooth $f(\pi)$. These distinct sources of variability should not be confused with one another. 

Limiting regimes of variability are best understood in terms of the network size $n$. If $f(\pi)$ is fixed and smooth, Theorem~\ref{thm:repro} asserts that the Binomial kernel $g_n(\pi;k)$ isolates the value of $f(\pi)$ at $\pi = k / (n\mu)$ as $n$ increases. If the support of $f(\pi)$ shrinks in $n$, the reverse may occur; in this case, a Poisson limiting distribution will be recovered, as is apparent from Theorem~\ref{thm:powerLaws}.

To understand how these limiting modes of behavior are achieved, we first categorize network behavior through the total number $E_n$ of expected edges, which will typically scale as some polynomial in $n$; see, e.g., \citet{bollobas2009metric}. Networks with $ E_n=\Theta(n^2)$ are said to be \emph{dense}, as in Section \ref{sec:net-pop}. Networks for which $E_n=\Theta(n^{2(1-\gamma)})$ with $0 < \gamma < 1/2$ are termed \emph{sparse}, while $\gamma=1/2$ yields \emph{extremely sparse} networks having $E_n=\Theta(n)$. The power-law networks studied in Section~\ref{sec:power-law} have $E_n = \Theta( \, \theta_n^2 n^{2(1-\gamma)} \, )$, and hence cover the entire range of sparse regimes.

In studying sparse networks, we see that a continuum of behaviour can be realized, ranging from very concentrated to very heterogeneous degrees. Degree distributions from such networks may exhibit a range of variability, and may reproduce $f(\cdot)$ over some range of $k$. Understanding this continuum of behaviour requires exercising direct control over the scaling of each variate $\pi$, as this in turn determines the sparsity of the network. 

\subsection{Controlling network sparsity}

As a tool to move smoothly through different network behaviors, we introduce the notion of $n$-scaling, using an $n$-dependent affine transformation $\pi \mapsto \pi(n)$ to obtain network degrees $d^{(n)}$ from $\operatorname{Bernoulli}\big( \pi_i(n) \pi_j(n) \big)$ trials. If we control the moments of $\pi(n)$, then $\E\!\big( d^{(n)} \big)$ and $\var\!\big( d^{(n)} \big)$ follow directly as a corollary of Proposition~\ref{thm:hierModels}.

\begin{corollary}[Sparse moments]
For $F(\pi)$ on $[0,1]$, define the scaling
\begin{equation*}
\pi(n) = \big( \sqrt{ \zeta / n^{2\gamma} } \, \pi + \sqrt{ \zeta' / n^{\smash{2\gamma'}} } \,\, \big) \wedge 1,
\end{equation*}
parameterized by nonnegative constants $(\gamma, \zeta)$ and $(\gamma', \zeta')$, with $\gamma_{\wedge} = \gamma \wedge \gamma'$. The mean and variance of each resultant network degree $d^{(n)}\!$ then behave as
\begin{align*}
 \E \!\big( d^{(n)} \big) & = n^{1-2\gamma} \big( \sqrt{\zeta} \mu + \sqrt{\zeta'} n^{\gamma-\gamma'}\big)^2 + \mathcal{O}\!\left( n^{-2\gamma_{\wedge}} \right),
\\ \var \!\big( d^{(n)} \big) & = n^{2(1-2\gamma)} \big(\sqrt{\zeta}
\mu \!+\! \sqrt{\zeta'} n^{\gamma-\gamma'}\big)^2
\left( \zeta \sigma^2 \!+\! n^{2\gamma-1} \right) \!+\! \mathcal{O}\!\left( n^{-2\gamma_{\wedge}} \!+\! n^{1-4\gamma_{\wedge}} \right).
\end{align*}
\end{corollary}

This corollary allows us to quantify the basic properties of $d^{(n)}$ in any scaling regime, as we can independently shift and scale $\pi$ via the map $\pi(n)$. The case $\gamma_{\wedge} = 1/2$ admits a limiting distribution, where the mean and variance of each $d^{(n)}$ will converge to a limit as $n$ grows. Setting $\gamma' = 1/2$ and $\gamma \rightarrow \infty$, by contrast, will recover the Poisson degree setting as $n \rightarrow \infty$.

Importantly, we may also achieve non-Poisson limits. If $\gamma_{\wedge} = 1/2$, then the dispersion $\var(d^{(n)}) / \E(d^{(n)})$ is given to leading order by $1 + \zeta \sigma^2 / n^{1-2\gamma}$. This quantity converges to $1$ when $\gamma>1/2$, but when $\gamma = 1/2$, it tends to $1 + \zeta \sigma^2$. Since the dispersion of a Poisson random variable is $1$, we immediately see that when $\gamma=1/2$, the limiting distribution of $d^{(n)}$ cannot be Poisson.

In this setting, the product $\zeta \sigma^2$ characterizes the limiting over-dispersion of $d^{(n)}$ relative to a Poisson variate. Note that $\zeta$ and $\sigma^2$ cannot both be determined from a single network observation (though their product can); instead, to estimate them from data would require multiple realizations of the same network generating mechanism at different sample sizes $n$.

These expressions also highlight that $1/n^{1-2\gamma}$ can give a natural rescaling of $d^{(n)}$ to enable convergence in distribution. This aligns with the deterministic power-law setting of Theorem~\ref{thm:powerLaws}, where the same rescaling is natural when $\gamma < 1/2$ and all degrees are growing in $n$. Overall degree magnitudes can thus grow at a rate matched to the deterministic setting, but individual and collective degree heterogeneity will depend on the variance of $\pi(n)$.

\subsection{Results for sparse networks}

To refine our understanding beyond the moments $\E(d^{(n)})$ and $\var(d^{(n)})$, we next simplify the expression for $\P(d^{(n)}\! = k)$, complementary to the exact result of Proposition~\ref{thm:hierModels}. For clarity of exposition, we take $\zeta' = 0$ and study $d^{(n)}\!$ under the rescaling $\pi(n) = \sqrt{\zeta / n^{2\gamma}} \, \pi$. Theorem~\ref{thm:repro} extends immediately to this setting, yielding a simplified expression for $\P(d^{(n)}\! = k)$ as per the following corollary.

\begin{corollary}[Sparse networks with polynomial degree growth]
Fix $\zeta > 0$ and $\gamma \in (0,1/2)$, and define $\pi(n) = \sqrt{\zeta / n^{2\gamma}} \, \pi$ for all $n \geq 2 \vee \zeta^{1/(2\gamma)}$. Theorem~\ref{thm:repro} then holds, with every instance of $\mu$ replaced by $\mu_n = \mu \zeta / n^{2\gamma}$:
\begin{equation}
\label{eq:SparseDegs}
\textstyle n \mu_n \P(d^{(n)}\! = k) = f\!\left( \frac{(k+1)\,\iota_{k,n}(\mu_n)}{(n+1)\mu_n} \right) I_{\mu_n}\!(k\!+\!1, n\!-\!k) \, \left\{ 1 + \mathcal{O}\!\left( \frac{1}{n \mu_n} \right) \right\}.
\end{equation}
\end{corollary}
We first note that fixing $\zeta = 1$ and letting $\gamma \rightarrow 0$ implies that $\mu_n \rightarrow \mu$, and thus Theorem~\ref{thm:repro} may be recovered exactly. Here, however, $\E\!\big( \pi(n) \big) = \sqrt{\mu_n \mu}$ is the geometric mean of $\mu_n$ and $\mu$, and thus $\mu_n$ in~\eqref{eq:SparseDegs} is \emph{not} the expectation of $\pi(n)$. Rather, the quantity $n \mu_n$ describes the effective range of $d^{(n)}$. Indeed, we see from the left-hand side of~\eqref{eq:SparseDegs} that $\P(d^{(n)}\! = k) $ must be rescaled by its effective range in order to be correctly normalized.

In contrast to the dense network regime of $\gamma = 0$, any $\gamma > 0$ will lead to a sparse network. In fact, as $\gamma$ exceeds $1/4$ and we move toward the regime of extremely sparse networks, $\P(d^{(n)}\! = k) $ begins to shift from a mixed Binomial toward a mixed Poisson distribution. This \emph{appreciably sparse} regime admits further simplification as described by the following corollary of Theorem~\ref{thm:repro}.

\begin{corollary}[Appreciably sparse networks]
Assume the setting of the previous corollary, but further restrict $\gamma \in (1/4,1/2)$, so that $\mu_n = \mu \zeta / n^{2\gamma} = o(n^{-1/2})$. Then for $k + 1 \leq (n+1) \mu_n$, the distribution of each degree satisfies
\begin{equation*}
\textstyle n \mu_n \P(d^{(n)}\! = k) = f\!\left( \frac{ (k+1)\,\rho_k(n\mu_n) }{ (n+1)\mu_n } \right) P(k+1,n \mu_n) \, \left\{ 1 + \mathcal{O}\!\left( n^{2\gamma-1} + n^{1-4\gamma} \right) \right\},
\end{equation*}
where $P(k+1,n \mu_n)$ is the regularized lower incomplete Gamma function, recognizable as the law of a $\operatorname{Poisson}(n\mu_n)$ random variable, and $\rho_k(n\mu_n) \in (0, 1)$ is defined via a $\operatorname{Poisson}(n\mu_n)$ random variable $Y_n$ as
\begin{equation*}
\rho_k(n\mu_n) = 1 - \frac{ \P(Y_n = k+1) }{ \P(Y_n \geq k+1)} \quad (k = 0, 1, \ldots ).
\end{equation*}
As in Theorem~\ref{thm:repro}, the argument of $f(\cdot)$ increases in $k$ from $0$ to $1$.
\end{corollary}

\begin{proof}
First, note that together the assumptions $k + 1 \leq (n + 1) \mu_n $ and $\mu_n = o(n^{-1/2})$ imply $k = o(\sqrt{n})$. Starting from the result of the previous corollary as given by~\eqref{eq:SparseDegs}, Lemma~\ref{thm:ItoPbnd} then allows us to replace $I_{\mu_n}\!(k+1, n-k)$ with $P(k+1,n\mu_n)$, and likewise under these conditions, Lemma~\ref{thm:ItoPRatiobnd} allows us to replace $\iota_{k,n}(\mu_n)$ with $\rho_k(n\mu_n) = P(k+2,n\mu_n) / P(k+1,n\mu_n)$:
\begin{equation*}
\textstyle n \mu_n \P(d^{(n)}\! = k) = f\Big( \frac{ (k+1)\,\rho_k(n\mu_n) \smash{ \myl\{ 1 + \mathcal{O}\left( k\mu_n + n \mu_n^2 \right) \myr\} } }{ (n+1)\mu_n } \Big) P(k+1,n \mu_n) \big\{ 1 + \epsilon_{k,n}(\mu_n) \big\},
\end{equation*}
where $\epsilon_{k,n}(\mu_n) = \mathcal{O}\!\left( 1/(n \mu_n) + k^2 / n + k\mu_n + n \mu_n^2 \right)$.

Next, we Taylor-expand $f$ about the desired point $\mu_n^{-1} (k+1)\rho_k(n\mu_n) / (n+1)$. Since $f$ is assumed nonzero on $[0,1]$ and with bounded first derivative, Lagrange's form of the remainder allows us to bound the resultant error. We see from the above that this yields an additional multiplicative term of the form $1 + \mathcal{O}\!\left(\mu_n^{-1} (k+1)\rho_k(n\mu_n) ( k\mu_n + n \mu_n^2) / (n+1) \right)$. Since $k + 1 \leq (n + 1) \mu_n $, with $k = o(\sqrt{n})$ and $\mu_n = o(n^{-1/2})$, we may subsume error terms in $k^2/n$ and $k \mu_n$ into $n \mu_n^2$. We obtain the result as stated by substituting $\mu_n \propto n^{-2\gamma}$.
\end{proof}

We can also achieve a more nuanced understanding for larger values of $k$.
Lemma~\ref{Beta-approx-a} shows that for general $k$, setting $n_k = n - k - 1$, we have
\begin{equation*}
\textstyle n_k \mu_n \P(d^{(n)}\! = k) 
= f\!\left( \frac{ (k+1)\,\rho_k(n_k\mu_n) }{ n_k\mu_n } \right) \! P(k+1, n_k \mu_n) \, \frac{\Gamma(n_k+k+1) }{\Gamma(n_k+1) \, (n_k)^k} \left\{ 1 + \epsilon_n(\gamma) \right\}
\end{equation*}
where $\epsilon_n(\gamma) = \mathcal{O}\!\left( n^{2\gamma-1} + n^{1-4\gamma} \right)$. When $k = o(\sqrt{n})$, then by Lemma~\ref{thm:gammaRatio}, the ratio $\Gamma(n_k+k+1) / \{ \Gamma(n_k+1) \, (n_k)^k\}$ tends to unity, and we recover the above corollary as stated. The latter formulation, however, explicitly takes into account the censoring effects that become prominent as $k = \Omega(\sqrt{n})$, in contrast to the above corollary. It reproduces a smooth transition between the regularized incomplete Beta and Gamma functions, especially for values of $k$ for which the ratio $\Gamma(n_k+k+1) / \{ \Gamma(n_k+1) \, (n_k)^k\}$ is far from unity.

The above corollaries highlight the phase transition that occurs at $\gamma = 1/4$. Comparing their statements, we see that the multiplication of $f(\cdot)$ by either an incomplete Beta or Gamma function corresponds to a censoring of degrees for large $k$. From Lemma~\ref{lem:BinomSurvival}, which characterizes the Binomial survival function, we see that exact reproduction of $f$ can only occur when $0 \leq n \mu_n - k = \omega\big( \sqrt{n} \big)$. Since $n \mu_n$ scales as $n^{1-2\gamma}$, such a range of $k$ exists only when $\gamma < 1/4$. Thus in the appreciably sparse case, \emph{all} values of $k$ are affected by $P(k+1,n \mu_n)$, and there is no region of perfect reproduction. 

\subsection{Results for extremely sparse networks}

As $\gamma \rightarrow 1/2$, we approach the regime of extreme sparsity. This serves as a model for cases in which degree heterogeneity has saturated, so that $\P(d^{(n)}\! = k)$ does not change appreciably as the network scales to larger sizes. This implies a limiting variate $d^{(\infty)}$, and so we must consider both the case of $n$ large but finite, as well as the formally infinite setting. To this end, the following corollary of Theorem~\ref{thm:repro} shows that for $\pi(n) = \sqrt{\zeta / n} \, \pi$, we may bound the behavior of $\P(d^{(n)}\! = k)$ and $\P(d^{(\infty)}\! = k)$ in terms of the fixed quantity $\zeta$.

\begin{corollary}[Extremely sparse networks]
Assume the setting of the previous corollary, fixing $\gamma = 1/2$ so that $\mu_n = \mu \zeta / n$. Then for $k = o(\sqrt{n})$, the distribution of each network degree $d^{(n)}\!$ admits the approximation
\begin{equation*}
\Bigg| \frac{ \mu \zeta \P(d^{(n)}\! = k) }{ f\big( \frac{ (k+1)\,\rho_k(\mu \zeta)}{ \mu \zeta } \big) P(k+1,\mu \zeta) \, \myl\{ 1 + \mathcal{O}\myl( \frac{ k^2 }{ n } \myr) \myr\} } - 1 \Bigg|
< \frac{c}{\zeta} \quad (0 \leq k \leq n - 1),
\end{equation*}
where $c = \sup_{\pi \in (0,1)} \left| f''(\pi) \right| / \left( 2\mu \inf_{\pi \in (0,1)} \left| f(\pi) \right| \right)$. Moreover, $d^{(n)}\!$ converges in distribution to a random variable $d^{\smash{(\infty)}}\!$ such that for some constant $\epsilon$,
\begin{equation*}
{ \textstyle
\mu \zeta \P(d^{(\infty)}\! = k) = f\!\left( \frac{ (k+1)\,\rho_k(\mu \zeta)}{ \mu \zeta } \right) P(k+1,\mu \zeta) \left\{ 1 + \epsilon \right\}, } \,\,\,
 \left| \epsilon \right| < \frac{c}{\zeta} \quad (k \geq 0).
\end{equation*}
\end{corollary}

\begin{proof}
We start from Theorem~\ref{thm:repro}, adapted to a scaling of $\mu_n = \mu \zeta / n $:
\begin{equation*}
 \bigg| \frac{ \mu \zeta \P(d^{(n)}\! = k) }{ f\big( \frac{(k+1)\,\iota_{k,n}(\mu_n)}{(n+1)\mu_n} \big) I_{\mu_n}\!(k\!+\!1, n\!-\!k) } - 1 \bigg|
< \frac{c}{\zeta} \, \frac{(k+1)\,\iota_{k,n}(\mu_n)}{(n+1)\mu_n} < \frac{c}{\zeta},
\end{equation*}
For $k = o(\sqrt{n})$, we may replace $I_{\mu_n}\!(k\!+\!1, n\!-\!k)$ with $P(k+1,\mu \zeta)$ via Lemma~\ref{thm:ItoPbnd}, and likewise $\iota_{k,n}(\mu_n)$ with $\rho_k(\mu \zeta)$ as per Lemma~\ref{thm:ItoPRatiobnd}:
\begin{equation*}
\bigg| \frac{ \mu \zeta \P(d^{(n)}\! = k) }{ f\Big( \frac{ (k+1) \, \rho_k(\mu \zeta) \smash{ \myl\{ 1 + \mathcal{O}\myl( (k + \mu \zeta) \mu_n \myr) \myr\} } }{ (n+1)\mu_n } \Big) P(k\!+\!1,\mu \zeta) \big\{ 1 \!+\! \mathcal{O}\big( \frac{k^2}{n} \!+\! (k \!+\! \mu \zeta) \mu_n \big) \big\} } - 1 \bigg|
 < \frac{c}{\zeta}.
\end{equation*}
By the same argument as in the preceding proof, we may Taylor-expand $f$ about the desired point $(k+1)\rho_k(\mu \zeta) / (\mu \zeta)$, and since $f$ is assumed nonzero on $[0,1]$ and with bounded first derivative, we obtain a multiplicative error that can be subsumed directly into the existing error terms:
\begin{equation*}
\bigg| \frac{ \mu \zeta \P(d^{(n)}\! = k) }{ f\big( \frac{ (k+1)\,\rho_k(\mu \zeta) }{ \mu \zeta } \big) P(k\!+\!1,\mu \zeta) \big\{ 1 \!+\! \mathcal{O}\big( \frac{k^2}{n} \!+\! (k \!+\! \mu \zeta ) \mu_n \!+\! \frac{1}{n}\big) \big\} } - 1 \bigg|
 < \frac{c}{\zeta}.
\end{equation*}
Since $k = o(\sqrt{n})$ and $\mu_n \propto n^{-1}$, the error term in $k^2/n$ will dominate, and thus we obtain the first stated result. Fixing $k$ and letting $n \rightarrow \infty$ then yields the second result. It may also be proved directly as per Theorem~\ref{thm:repro}, by applying a Taylor series about the mean of a truncated $\operatorname{Gamma}(k\!+\!1,1)$ random variable, and then appealing to the bounded convergence theorem to obtain $ \P( d^{(\infty)} \!= k ) = \lim_{n \rightarrow \infty} \P( d^{(n)} \!= k )$ for every fixed $k = 0,1,\ldots$.
\end{proof}

This corollary provides a bound on the relative error between $\mu \zeta \P(d^{(n)}\! = k)$ and $f\big( (k+1) \rho_k(\mu \zeta) / (\mu \zeta) \big) P(k+1,\mu \zeta) $ when $k = o(\sqrt{n})$. In fact, by our earlier arguments in the appreciably sparse regime, whenever $k$ grows with $n$, the masses $\P(d^{(n)}\! = k)$ are already rapidly approaching zero. We see that this relative error can nevertheless be made arbitrarily small by choosing $\zeta$ large relative to the smoothness of $f(\cdot)$, as encapsulated by the constant $c$. 

Viewed for fixed $\zeta$, however, this result does not constrain the form of $\P(d^{(n)}\! = k)$. When $\zeta$ is very small, then \eqref{infinite-Poisson-mixture} concentrates near $k=0$ in a manner we shall formalize below. As $\zeta $ increases, the Poisson kernel giving rise to $\P(d^{(n)}\! = k)$ approaches a Normal density, analogously to the analysis of \citet{hald1968mixed} discussed earlier. Thus reproduction of $f(\cdot)$ will gradually be achieved for increasing $\zeta$---as implied by the bound of the corollary. For values of $\zeta$ in between these extremes, a range of different behavior is possible.

Another way to arrive at this understanding is to start from Proposition~\ref{thm:hierModels}. Recalling~\eqref{eq:margProbD}, we view $\P(d = k)$ as $\E_\pi\!\big\{ g_n(\mu \pi,k) \big\}$:
\begin{equation*}
\P(d = k) = \int g_n(t,k) {\textstyle \, dF\big( \frac{t}{\mu} \big) }
= \int g_n(\mu t,k) \, dF(t) = \E_\pi\!\big\{ g_n(\mu \pi,k) \big\},
\end{equation*}
so that the Binomial kernel $g_n(\mu \pi_i,k)$ is interpreted as the conditional distribution $\P(d_i = k \,\vert\, \pi_i )$. From this, we see immediately that 
\begin{equation*}
\P( d^{(n)}\! = k ) = \E_{\pi(n)}\!\Big\{ g_n\big( \E\!\left[\pi(n)\right] \pi(n),k\, \big) \Big\} = \E_\pi\!\big\{ g_n( \mu \zeta \pi / n,k) \big\},
\end{equation*}
with the Binomial kernel $g_n( \mu \zeta \pi / n,k)$ converging to a Poisson kernel. Indeed, approximating $g_n( \mu \zeta \pi / n,k)$ for any fixed $k$, we may write \begin{equation}
\label{infinite-Poisson-mixture} 
\P(d^{(n)}\! = k) = \int \frac{(\mu \zeta t)^k e^{- \mu \zeta t}}{k!} \, dF(t) + \textstyle \mathcal{O}\!\left(\frac{1}{n}\right) \quad (\text{$k \geq 0$ fixed}).
\end{equation}
In other words, the $\operatorname{Binomial}(n-1,\mu \zeta \pi_i / n)$ conditional distribution of the $i$th degree $d_i^{\smash{(n)}} \,\vert\, \pi_i $ is converging to a $\operatorname{Poisson}(\mu \zeta \pi_i)$ distribution. Marginally, we see that this yields a mixed Poisson distribution, right-truncated at unity. 

We further observe from~\eqref{infinite-Poisson-mixture} that
\begin{equation*}
\P(d^{(n)}\! = 0)
= \int e^{-\mu \zeta t} \, dF(t) + \textstyle \mathcal{O}\!\left(\frac{1}{n}\right), 
\end{equation*}
and so, if $F(\cdot)$ admits a density $f(\cdot)$, the decay of $\P(d^{(n)}\! = 0)$ in $\mu \zeta$ depends on the smoothness of $f(\cdot)$. For $\mu \zeta < 1$ and $f(\cdot)$ sufficiently smooth, we have
\begin{equation*}
\P(d^{(n)} = 0)
\rightarrow \int_{0}^{1} \big\{ 1 - \mu \zeta t + \mathcal{O}(\zeta^2) \big\} f(t) \,dt 
= 1 - \mu^2 \zeta + \mathcal{O}(\zeta^2).
\end{equation*}
This expression shows that a larger $\mu$ implies slower convergence in $\zeta$ of $\P(d^{(n)} = k)$ to a distribution consisting of mass only at $k=0$. The variable $\zeta$, by contrast, serves to stretch out the probability mass function of $d^{(n)}$ along the real line. When $n$ is only slightly larger than $\zeta$, networks generated from this model will be very dense, but they will rapidly saturate as $n \gg \zeta$.

From~\eqref{infinite-Poisson-mixture}, we recognize $\P(d^{(\infty)}\! = k)$ in the limit as $n \rightarrow \infty$, remarking again that it takes the form of a truncated mixed Poisson distribution. Formally, $d^{(\infty)}$ arises from an infinite exchangeable random graph \citep{janson2008graph}. In this setting, we can view any degree $d_i^{\smash{(\infty)}} \,\vert\, \pi_i $ as a conditionally $\operatorname{Poisson}( \mu \zeta \pi_i)$ random variable, given its parameter $\pi_i$. The natural limit of the scaling regimes in Theorems~\ref{thm:powerLaws} and~\ref{thm:repro} can therefore be attained, and indeed coincides with the more general asymptotic analysis of \citet[Theorem~3.13]{bollobas2007phase}.

Further insight can be obtained from the case in which $f(\cdot)$ is uniform on $[0,1]$. Then each $\pi(n)$ is a uniform variate whose range $[0, \sqrt{\zeta / n} ]$ is shrinking toward zero at rate $1 / \sqrt{n}$. In this instance $\P(d^{(\infty)} \!= k)$ decays monotonically as $k$ increases from zero; $d^{(\infty)}$ has mixing density $f(\pi) = \I(0 \leq \pi < 1)$ dilated by $(\zeta/2)^{-1}$, and so $ \P(d^{(\infty)} \!= k) = (\zeta/2)^{-1} P(k\!+\!1,\zeta/2) $. This agrees not only with the corollary above---since $c=0$ in this case---but also with the results of \citet{bollobas2007phase} for general mixed Poisson degrees.

Finally, we note that as a Poisson mixture, $d^{(\infty)}$ is over-dispersed relative to a simple Poisson variate. When $f(\cdot)$ is uniform, the dispersion of $d^{(\infty)}$ evaluates to $1 + \zeta /12$. This depends linearly on $\E\big( d^{(\infty)} \big) = \zeta / 4$, and so increasingly larger values of $\zeta$ lead to increasingly variable realizations of $d^{(\infty)}$ relative to a Poisson limiting regime. In fact, as $\zeta $ increases, the above corollary shows that $\P(d^{(\infty)} \!= k)$ is increasingly near to the discrete uniform distribution on $\{0,\ldots,\lfloor \zeta / 2 \rfloor\}$, rather than the strongly decaying limiting distribution achieved for moderate $\zeta$. In this way, the effects of reproduction and sampling combine to yield a plethora of possible limiting forms.

\section{Conclusions}

The above results provide a foundational analysis of degree-based network models, in which the structure of realized networks is governed by properties of their degree sequences. First, we have seen that it is possible to characterize the limiting extremes of variation achievable through such models, by analyzing the complementary roles played by deterministic and random specifications. Second, we have established exact results and large-sample approximations for power-law networks and other more general forms, including a central limit theorem for weights whose pairwise products parameterize independent Bernoulli trials. Finally, we have achieved a thorough understanding of network populations parameterized by smooth distributions, across a range of realistic sparsity regimes.

Because networks are naturally summarized through their degrees, these conclusions have important implications for practitioners. Crucially, they highlight that variation explained through expected degree structure should not automatically be attributed to more complicated generative mechanisms. For example, our analysis has provided a theoretical explanation of an empirically observed exponential cutoff effect in power law networks \citep{newman2001structure}, showing that it can be explained by sampling variability alone. More generally, the quantification of sampling variability both within and across network populations is set to increase in importance, as replications and time series of network observations become more widely available.

We have also introduced formal mechanisms that us allow to increase heterogeneity and to create greater degree structure diversity. This means that degree-based models are richer than they may at first appear; they also provide a natural first step by which to approach more realistic network models. In fact, it is reasonable to compare the form of multiplicative model studied here with the simple linear regression model for Normal observations, as both quantify the first-order structure of the data. Once network degree characteristics have been fitted, moving beyond main effects will require richer classes of model structure and hierarchies of model properties. The degree-corrected blockmodel \citep{zhao2012consistency} is an important example. Similarly, degree-based network models such as those studied here are naturally amenable to Bayesian inference and hierarchical modeling.

Looking to the future, we must understand how additional structure in network models affects the properties of observed network degrees. The results presented here establish the regimes of variability that are attainable by the simplest nontrivial statistical models. With such models we can then understand the variability of network summaries, such as the empirical degrees themselves. We now have a clear understanding of the properties of estimators of model parameters based on these summaries. Going beyond the multiplicative model structure analyzed here, and understanding joint as well as marginal properties of observed network degrees, will be an important next step in moving the field of statistical network analysis forward.

\appendix

\section{Auxiliary lemmas}
\label{sec:auxLemmas}

\begin{lemma}\label{lem:deltaMethod}
$\!\!\!\!\!$In the setting of Theorem~\ref{thm:CLT}, $\sqrt{\|\underline{d}\|_1} \!=\! \sqrt{\E\left(\|\underline{d}\|_1 \vert \underline{\pi}\right)} + \mathcal{O}_P(1)$.
\end{lemma}

\begin{proof}
The arguments establishing~\eqref{eq:DegConv} show $\E(d_i \,\vert\, \underline{\pi}) \rightarrow \infty$ implies
\begin{equation*}
\frac{\| \underline{d} \|_1}{\E( \| \underline{d} \|_1 \,\vert\, \underline{\pi})} \overset{P}{\longrightarrow} 1
\quad \text{and} \quad \frac{ \| \underline{d} \|_1 - \E( \| \underline{d} \|_1 \,\vert\, \underline{\pi})}{\sqrt{\var( \| \underline{d} \|_1 \,\vert\, \underline{\pi})}} \overset{L}{\longrightarrow} \operatorname{Normal}(0,1).
\end{equation*}
Moreover, $\|\underline{d}\|_1 / \E(\|\underline{d}\|_1
\,\vert\, \underline{\pi}) - 1$ normalized by its standard deviation is $\mathcal{O}_P(1)$:
\begin{equation*}
\frac{ \|\underline{d}\|_1 }{\E\big(\|\underline{d}\|_1 \,\vert\, \underline{\pi}\big)} = 1
+ \mathcal{O}_P \left( \frac{\sqrt{\var( \|\underline{d}\|_1 \,\vert\, \underline{\pi}
)}}{\E(\|\underline{d}\|_1 \,\vert\, \underline{\pi}) } \right),
\end{equation*}
with $\E (d_i\,\vert\, \underline{\pi}) \rightarrow \infty$ implying $ \sqrt{\var( \|\underline{d}\|_1 \,\vert\, \underline{\pi} )} / \E(\|\underline{d}\|_1 \,\vert\, \underline{\pi}) \rightarrow 0$. Since the square root function has continuous derivatives at $1$, this implies that whenever $\E (d_i\,\vert\, \underline{\pi})$ is growing
in $n$, we may expand $\sqrt{\|\underline{d}\|_1} / \sqrt{\E(\|\underline{d}\|_1 \,\vert\,
\underline{\pi})}$ in a convergent Taylor series about $1$ as follows:
\begin{align*}
\sqrt{\frac{ \|\underline{d}\|_1 }{\E\left(\|\underline{d}\|_1 \,\vert\, \underline{\pi}\right)}}
& = 1 + \frac{1}{2} \!\left( \!\frac{ \|\underline{d}\|_1 }{\E\left(\|\underline{d}\|_1 \,\vert\,
\underline{\pi}\right)} \!-\! 1 \! \right)
+ o_P \!\left(\! \frac{\sqrt{\var( \|\underline{d}\|_1 \,\vert\, \underline{\pi} )}}{\E(\|\underline{d}\|_1 \,\vert\, \underline{\pi}) } \right)
\\ \Rightarrow \sqrt{\|\underline{d}\|_1} - \sqrt{\E\left(\|\underline{d}\|_1 \,\vert\, \underline{\pi}\right)}
& = \sqrt{\frac{\var\left(\|\underline{d}\|_1 \,\vert\, \underline{\pi}\right)}{\E\left(\|\underline{d}\|_1 \,\vert\, \underline{\pi}\right)}}
\left\{ \frac{ \|\underline{d}\|_1 - \E\left(\|\underline{d}\|_1 \,\vert\,
\underline{\pi}\right) }{2\sqrt{\var\left(\|\underline{d}\|_1 \,\vert\, \underline{\pi}\right)}}
+ o_P (1) \right\} \!.
\end{align*}
Since whenever $\E (d_i\,\vert\, \underline{\pi}) \rightarrow \infty$, $\sqrt{\var\left(\|\underline{d}\|_1 \,\vert\, \underline{\pi}\right) /
\E\left(\|\underline{d}\|_1 \,\vert\, \underline{\pi}\right)} = \mathcal{O}(1)$ and $\left\{ \|\underline{d}\|_1 - \E\left(\|\underline{d}\|_1 \,\vert\,
\underline{\pi}\right) \right\} / \sqrt{\var\left(\|\underline{d}\|_1 \,\vert\, \underline{\pi}\right)}$ converges in law to a standard Normal, we conclude that $\sqrt{\|\underline{d}\|_1} - \sqrt{\E\left(\|\underline{d}\|_1
\,\vert\, \underline{\pi}\right)}$ is bounded in probability.
\end{proof}

\begin{lemma}\label{thm:gammaRatio}
For $z > \beta \geq 0$, we have the approximations
\begin{align*}
\frac{\Gamma(z)}{\Gamma(z - \beta)}
& = \left( z - \beta \right)^\beta \left\{ 1 + \frac{\beta(\beta - 1)}{2z} + \frac{(3\beta+2)(\beta+1)\beta(\beta - 1)}{24z^2} + \mathcal{O}\!\left( \frac{\beta^6 }{ z^3 } \right) \right\}
\\ & = \qquad\,\,\, z^{\beta} \left\{ 1 - \frac{(\beta + 1)\beta}{2z} + \frac{(3\beta+2)(\beta+1)\beta(\beta - 1)}{24z^2} + \mathcal{O}\!\left( \frac{\beta^6 }{ z^3 } \right) \right\},
\end{align*}
with the magnitude of each leading term given by $(\beta^2 / 2z)^i / i!$ for $i = 0, 1, \ldots$.
\end{lemma}

\begin{proof}
A general asymptotic expansion for ratios of Gamma functions is given by \citet{tricomi1951asymptotic}. We specialize to the case of a convergent asymptotic expansion for $z > \beta \geq 0$ as follows. First, a convergent version of Stirling's formula for $z > 0$ \citep{whittaker1927course} yields
\begin{equation*}
\log \Gamma(z) = \left(z-\frac{1}{2}\right) \log z - z + \log \sqrt{2\pi} + \frac{1}{12(z+1)} + \frac{1}{12(z+1)(z+2)} + \mathcal{O}(z^{-3}),
\end{equation*}
and so for fixed $\beta \geq 0$ such that $z > \beta$ we may write
\begin{equation*}
\log \frac{\Gamma(z)}{\Gamma(z-\beta)}
= \beta \log (z-\beta) - \beta - \frac{\beta}{12(z+1)^2} + \left(z-\frac{1}{2}\right) \log \frac{z}{z-\beta} + \mathcal{O}(z^{-3}).
\end{equation*}
Exponentiating both sides of this expression, we obtain
\begin{equation*}
\frac{\Gamma(z)}{\Gamma(z-\beta)}
= \left( z-\beta \right)^\beta e^{-\beta - \frac{\beta}{12(z+1)^2}} \left(1 - \frac{\beta}{z} \right)^{-\left(z-\frac{1}{2}\right)} \left\{ 1 + \mathcal{O}(z^{-3}) \right\}.
\end{equation*}
To obtain the first stated result, we apply the convergent Taylor expansion
\begin{equation*}
e^{ -\beta - \frac{\beta}{12(z+1)^2} - z \log \left(1 - \frac{\beta}{z} \right) }
= e^{ \, \frac{\beta^2}{2z} \myl\{ 1 + \mathcal{O}\myl( \frac{\beta }{ z } \myr) \myr\} }
= 1 + \frac{\beta^2}{2z} + \frac{3\beta^4 + 8\beta^3-2\beta}{24z^2} + \mathcal{O}(z^{-3}),
\end{equation*}
and multiply it by the following expansion, similarly convergent for $z > \beta$:
\begin{equation*}
\left(1 - \frac{\beta}{z} \right)^{\frac{1}{2}}
= 1 - \frac{\beta}{2z} - \frac{\beta^2}{8z^2} + \mathcal{O}(z^{-3}).
\end{equation*}
An expansion of $(1-\beta/z)^\beta$ then yields the second result from the first.

The general asymptotic expansion in the form written by \citet[Eqn.~4]{tricomi1951asymptotic} is $\Gamma(z+\alpha) / \Gamma(z) = \sum_{n=0}^{\infty} A_n(\alpha) z^{\alpha-n}$. To show that $A_n(\alpha)$ is a polynomial in $\alpha$ of degree $2n$, we appeal to strong induction as follows. We first construct a formal statement $P(n)$ for any fixed $n=0,1,\dots$: \emph{ $P(n)$: The coefficient $A_n(\alpha)$ takes the form of $A_n(\alpha)=\sum_{k=0}^{2n} \tilde{C}_{k,n} \alpha^{k}$, where $\tilde{C}_{k,n}$ is defined for $k=0,\dots, 2n$, and does not depend on $\alpha$.} 
From \citet[p.~135]{tricomi1951asymptotic}, we have directly that $A_0(\alpha)=1$ and $A_1(\alpha) = \binom{\alpha}{2}$. Thus statements $P(0)$ and $P(1)$ hold with $\tilde{C}_{0,0}=1$ and $\tilde{C}_{1,1}=1/2$. Now assume $P(0) \dots P(n-1)$ to hold for some $n\ge 2$. Using \citet[p.~137, Eqn.~10]{tricomi1951asymptotic} and $P(0)\dots P(n-1)$, we write
\begin{equation}
\label{eq:An}
A_{n}(\alpha) 
= \frac{1}{n}\sum_{m=0}^{n-1}\binom{\alpha\!-\!m}{n\!-\!m\!+\!1} A_m(\alpha)
= \frac{1}{n}\sum_{m=0}^{n-1}\binom{\alpha\!-\!m}{n\!-\!m\!+\!1} \sum_{k=0}^{2m} \tilde{C}_{k,m} \alpha^{k}.\!\!\!
\end{equation}
Since $\Gamma(\alpha-m+1) = (\alpha-m)(\alpha-m-1)\cdots (\alpha-n) \, \Gamma(\alpha-n)$, we have that
\begin{equation*}
\binom{\alpha-m}{n-m+1} \propto \frac{\Gamma(\alpha-m+1)}{\Gamma(\alpha-n)} = (b_0+\dots +b_{n-m}\alpha^{n-m}+\alpha^{n-m+1})
\end{equation*}
for some polynomial in $\alpha$ with coefficients $\{b_k\}$, with $b_{n-m+1}=1$. From~\eqref{eq:An},
\begin{align*}
A_{n}(\alpha)
& = \frac{1}{n}\sum_{m=0}^{n-1} \frac{ (b_0+\dots + b_{n-m}\alpha^{n-m} + \alpha^{n-m+1}) }{ \Gamma(n-m+2) } \sum_{k=0}^{2m} \tilde{C}_{k,m} \alpha^{k}
\\ & 
= \frac{1}{n} \sum_{m=0}^{n-1} \sum_{k=0}^{2m} \frac{ (\tilde{C}_{k,m} b_0 \alpha^{k} + \dots + \tilde{C}_{k,m} b_{n-m} \alpha^{n-m+k} + \tilde{C}_{k,m} \alpha^{n-m+1+k}) }{ \Gamma(n-m+2) }.
\end{align*}
We recognize a polynomial in $\alpha$, where for fixed $m$, the powers of $\alpha$ range from $\alpha^0$ to $\alpha^{n+m+1}$. Since $m \leq n-1$, the leading term has order $\alpha^{2n}$. Thus $P(0),\dots P(n-1)\Rightarrow P(n)$, and the hypothesis is proved by strong induction.

Having proved that $A_{n}(\alpha) = \sum_{k=0}^{2n} \tilde{C}_{k,n} \alpha^{k}$, we now set this equal to the double sum above and equate coefficients. Only one value of the double index $(k,m)$ attains the leading power $\alpha^{2n}$, and thus its coefficient satisfies $\tilde{C}_{2n,n} = \tilde{C}_{2(n-1),n-1} / (2n)$. Iterated application of this relationship establishes that
\begin{equation*}
\tilde{C}_{2n,n} = \left(\frac{1}{2n}\right) \frac{1}{2(n-1)} \frac{1}{2(n-2)} \cdots \tilde{C}_{2,1}
= \frac{1}{2^{n} n!}.
\vspace{-0.75\baselineskip}%
\end{equation*}
\end{proof}

\begin{lemma}\label{thm:VarBnd}
Let $I_\mu(k+1,n-k)$ denoted the regularized incomplete Beta function. For all $0 < \mu \leq 1$ and $n > k \geq 0$,
\begin{equation}
\label{eq:IncBetaRat}
\left( \frac{k+2}{n+2} \right) \frac{I_\mu(k+3,n-k)}{I_\mu(k+2,n-k)} - \left( \frac{k+1}{n+1} \right) \frac{ I_\mu(k+2,n-k) }{ I_\mu(k+1,n-k) } < \frac{1}{n+2}.
\end{equation}
\end{lemma}

\begin{proof}
Applying integration by parts to $I_\mu(k+2,n-k)$ establishes
\begin{equation}
\label{eq:IncBetaRecurs}
\frac{ I_\mu(k+2,n-k) }{ I_\mu(k+1,n-k) } = 1 - \frac{\mu^{k+1}(1-\mu)^{n-k}}{(k+1) B(\mu;k+1,n-k)},
\end{equation}
where $B(\mu; k+1, n-k) = \int_0^\mu t^{(k+1)-1} (1-t)^{(n-k)-1} \, dt$ is the incomplete Beta function. Substituting~\eqref{eq:IncBetaRecurs} into~\eqref{eq:IncBetaRat}, we see that~\eqref{eq:IncBetaRat} is equal to
\begin{equation*}
\frac{1}{(n\!+\!2)(n\!+\!1)} \!\left[ n \!-\! k \!+\!\mu^{k+1} (1\!+\!\mu)^{n-k} \! \left\{ \! \frac{n\!+\!2}{B(\mu; k\!+\!1,n\!-\!k)} \!-\! \frac{\mu(n\!+\!1)}{B(\mu; k\!+\!2,n\!-\!k)} \! \right\} \right]\!.
\end{equation*}
Since $B(\mu; k+2,n-k) < \mu \, B(\mu; k+1,n-k)$ for $0 < \mu \leq 1$, we see that
\begin{equation*}
\left( \frac{k\!+\!2}{n\!+\!2} \right) \frac{I_\mu(k\!+\!3,n\!-\!k)}{I_\mu(k\!+\!2,n\!-\!k)} - \left( \frac{k\!+\!1}{n\!+\!1} \right) \frac{ I_\mu(k\!+\!2,n\!-\!k) }{ I_\mu(k\!+\!1,n\!-\!k) }
< \frac{n \!-\!k + \frac{\mu^{k+1}(1-\mu)^{n-k}}{B(\mu; k+1,n-k)} }{(n\!+\!2)(n\!+\!1)}.
\end{equation*}
Appealing once again to~\eqref{eq:IncBetaRecurs}, the lemma follows from the upper bound
\begin{align*}
\frac{\mu^{k+1}(1-\mu)^{n-k}}{B(\mu; k\!+\!1,n\!-\!k)}
& = (k+1) \left( 1 - \frac{I_{\mu}(k\!+\!2,n\!-\!k)}{I_{\mu}(k\!+\!1,n\!-\!k)} \right)
\leq (k+1)(1-\mu) < k+1.
\end{align*}
\begin{equation*}
\vspace{-1.75\baselineskip}
\end{equation*}
\end{proof}

\begin{lemma}\label{thm:ItoPbnd}
Let $n > k \geq 0$ be integers, fix some $\delta \geq 0$, and let $I_{\mu_n}\!(k+\delta,n-k)$ and $P(k+\delta,n\mu_n)$ denote respectively the regularized incomplete Beta and Gamma functions. Then for $k = o(\sqrt{n})$ and $\mu_n = o(n^{-1/2})$,
\begin{equation*}
\textstyle I_{\mu_n}\!(k+\delta,n-k) = P(k+\delta,n\mu_n) \left\{ 1 + \mathcal{O}\!\left( \frac{k^2}{n} + k\mu_n + n \mu_n^2 \right) \right\}.
\end{equation*}
\end{lemma}

\begin{proof}
Let $\gamma(k+\delta,n\mu_n)$ be the lower incomplete Gamma function; then
\begin{equation*}
\frac{ I_{\mu_n}\!(k+\delta,n-k) }{ P(k+\delta,n\mu_n) }
= \frac{ \Gamma(n+\delta) }{ \Gamma(n-k) n^{k+\delta} } \frac{ \int_{0}^{n \mu_n} t^{k+\delta-1} \left( 1 - \frac{t}{n} \right)^{n-k-1} \,dt }{ \gamma(k+\delta,n\mu_n) }.
\end{equation*}
To show when this ratio is close to $1$, we apply a Taylor expansion to $\exp\!\left\{ \left( n-k-1 \right) \log \left( 1 - t/n \right) \right\}$. Lagrange's form of the remainder implies
\begin{equation*}
\left( 1 - \frac{t}{n} \right)^{n-k-1}
= e^{-t} \exp\left\{ (k+1)\frac{t}{n} - \left( \frac{n-k-1}{2} \right) \left(1 - \frac{ \tau(t) }{ n } \right)^{-2} \left(\frac{t}{n}\right)^2 \right\}
\end{equation*}
for some $\tau(t) \in (0,t)$, and so we obtain the following bound:
\begin{equation*}
\exp\left\{ -\frac{ n-k-1 }{ 2 \left(1 - \smash{\mu_n} \right)^2 } \, \mu_n^2 \right\}
< \frac{ \int_{0}^{n \mu_n} t^{k+\delta-1} \left( 1 - \frac{t}{n} \right)^{n-k-1} \,dt }{ \gamma(k+\delta,n\mu_n) } 
< \exp\!\left\{ (k+1)\mu_n \right\}.
\end{equation*}
Taylor expansions of these exponential bounding terms then imply that
\begin{equation*}
\frac{ I_{\mu_n}\!(k+\delta,n-k) }{ P(k+\delta,n\mu_n) }
= \frac{ \Gamma( n+\delta) }{ \Gamma(n-k) n^{k+\delta} } \left\{ 1 + \mathcal{O}\!\left( k\mu_n + (n-k) \mu_n^2 \right) \right\}.
\end{equation*}
Finally, appealing to Lemma~\ref{thm:gammaRatio} and recalling that $\delta$ is order one, we have
\begin{equation*}
\frac{ I_{\mu_n}\!(k+\delta,n-k) }{ P(k+\delta,n\mu_n) }
= \left(1 + \frac{\delta}{n}\right)^{k+\delta} \left\{ 1 + \mathcal{O}\!\left( \frac{k^2}{n} + k\mu_n + (n-k) \mu_n^2 \right) \right\}.
\end{equation*}
The stated result follows, using $k = o(\sqrt{n})$ to simplify this expression.
\end{proof}

\begin{lemma}\label{thm:ItoPRatiobnd}
Let $n > k \geq 0$ be integers, fix some $\delta \geq 0$, and let $I_{\mu_n}\!(k+\delta,n-k)$ and $P(k+\delta,n\mu_n)$ denote respectively the regularized incomplete Beta and Gamma functions. Then for $k = \mathcal{O}(\sqrt{n})$ and $\mu_n = o(n^{-1/2})$,
\begin{equation*}
 \frac{ I_{\mu_n}\!(k+1+\delta,n-k) }{ I_{\mu_n}\!(k+\delta,n-k) } = \frac{ P(k+1+\delta,n\mu_n) }{ P(k+\delta,n\mu_n) } \left\{ 1 + \mathcal{O}\!\left( k\mu_n + n \mu_n^2 \right) \right\}.
\end{equation*}
\end{lemma}

\begin{proof}
Let $\gamma(k+\delta,n\mu_n)$ be the lower incomplete Gamma function; then
\begin{multline*}
 \frac{ I_{\mu_n}\!(k+1+\delta,n-k) }{ I_{\mu_n}\!(k+\delta,n-k) }
\frac{ P(k+\delta,n\mu_n) }{ P(k+1+\delta,n\mu_n) }
= \left( 1 + \frac{\delta }{ n } \right)
\\ \cdot \frac{ \int_{0}^{n \mu_n} t^{k+1+\delta-1} \left( 1 - \frac{t}{n} \right)^{n-k-1} \,dt }{ \gamma(k+1+\delta,n\mu_n) } \frac{ \gamma(k+\delta,n\mu_n) }{ \int_{0}^{n \mu_n} t^{k+\delta-1} \left( 1 - \frac{t}{n} \right)^{n-k-1} \,dt },
\end{multline*}
where we recall that $\delta$ is order one. The results of Lemma~\ref{thm:ItoPbnd} then imply
\begin{equation*}
 \frac{ I_{\mu_n}\!(k\!+\!1\!+\!\delta,n\!-\!k) }{ I_{\mu_n}\!(k\!+\!\delta,n\!-\!k) }
\frac{ P(k\!+\!\delta,n\mu_n) }{ P(k\!+\!1\!+\!\delta,n\mu_n) }
= \left( 1 \!+\! \frac{\delta }{ n } \right) \left\{ 1 \!+\! \mathcal{O}\!\left( k\mu_n \!+\! (n\!-\!k) \mu_n^2 \right) \right\}.
\vspace{-0.75\baselineskip}%
\end{equation*}
\end{proof}

\begin{lemma}\label{Beta-approx-a}
Let $n > k \geq 0$ be integers, fix some $\delta \geq 0$, and let $I_{\mu_n}\!(k+\delta,n-k)$ and $P\big(k+\delta,(n-k-1)\mu_n\big)$ denote respectively the regularized incomplete Beta and Gamma functions. Then for $\mu_n = o(n^{-1/2})$, there exists some $\epsilon_{k,n}(\mu_n) \in [0, (n-k-1)\mu_n^2 (1-\mu_n)^{-2} / 2 \,)$ such that
\begin{align*}
I_{\mu_n}\!(k\!+\!\delta,n\!-\!k)
& \!=\!\frac{\Gamma(n+\delta) P(k+\delta,(n-k-1)\mu_n) }{\Gamma(n-k)(n-k-1)^{k+\delta}} \big\{ 1-\epsilon_{k,n}(\mu_n) \big\},
\\ \frac{ I_{\mu_n}\!(k\!+\!\delta\!+\!1,n\!-\!k) }{ I_{\mu_n}\!(k\!+\!\delta,n\!-\!k) } 
& \!=\!\frac{ P(k\!+\!\delta\!+\!1,(n\!-\!k\!-\!1)\mu_n) }{ P(k\!+\!\delta,(n\!-\!k\!-\!1)\mu_n) } { \textstyle \big( \frac{ n+\delta }{ n-k-1 } \big) } \!\left\{ 1 \!+\! \mathcal{O}\big( (n\!-\!k\!-\!1)\mu_n^2 \big) \right\}\!.
\end{align*}
\end{lemma}

\begin{proof}
For $\gamma\big(k+\delta,(n-k-1)\mu_n\big)$ the lower incomplete Gamma function,
\begin{equation*}
\frac{ I_{\mu_n}\!(k+\delta,n-k) }{ P(k+\delta,(n-k-1)\mu_n) }
= \frac{ \Gamma(n+\delta) }{ \Gamma(n-k) n^{k+\delta} } \frac{ \int_{0}^{n \mu_n} t^{k+\delta-1} \left( 1 - \frac{t}{n} \right)^{n-k-1} \,dt }{ \gamma\big(k+\delta,(n-k-1)\mu_n\big) }.
\end{equation*}
Applying the same Taylor expansion as in Lemma~\ref{thm:ItoPbnd}, we may write $(1-t/n)^{n-k-1} = \exp\{- t(n-k-1)/n\} \exp\{- v(t)\}$, for $v(t) \geq 0$ defined as
\begin{equation}
\label{eq:vt}
v(t) = \frac{n-k-1}{2 \big( 1 - \frac{\tau(t)}{n} \big)^2} \left( \frac{t}{n} \right)^2 \quad \left(0 < \frac{\tau(t)}{n} < \frac{t}{n} < \mu_n \right).
\end{equation}
If $n \mu_n^2 \rightarrow 0$, then we may apply a second Taylor series to $\exp\{-v(t)\}$ to obtain $\exp\{-v(t)\} = 1 - \exp\{-v'(t)\} v(t)$ for some $0 < v'(t) < v(t)$. Thus
\begin{multline*}
\frac{ \Gamma(n-k) I_{\mu_n}\!(k+\delta,n-k) }{ \Gamma(n+\delta) P(k+\delta,(n-k-1)\mu_n) }
= \frac{ \int_{0}^{n \mu_n} t^{k+\delta-1} e^{-(n-k-1)\frac{t}{n}} e^{-v(t)} \,dt }{ n^{k+\delta} \, \gamma\big(k+\delta,(n-k-1)\mu_n\big) }
\\ \hskip2.2cm = \frac{ \int_{0}^{n \mu_n} t^{k+\delta-1} e^{-(n-k-1)\frac{t}{n}} \big\{ 1 - e^{-v'(t)} v(t) \big\} \,dt }{ n^{k+\delta} \, \gamma\big(k+\delta,(n-k-1)\mu_n\big) }
\\ = \frac{ \int_{0}^{(n-k-1) \mu_n} u^{k+\delta-1} e^{-u} \big\{ 1 - e^{-v'\myl(t(u)\myr)} v\big(t(u)\big) \big\} \,du }{ (n-k-1)^{k+\delta} \, \gamma\big(k+\delta,(n-k-1)\mu_n\big) },
\end{multline*}
where the last line follows by letting $u = (n-k-1)t/n$. From~\eqref{eq:vt} we then have $0 \leq \exp\!\big\{ \!-\!v'\big(t(u)\big)\big\} \, v\big(t(u)\big) < (n-k-1)\mu_n^2 (1-\mu_n)^{-2} / 2$.
\end{proof}

\section*{Acknowledgements}
Work supported in part by the US Army Research Office under PECASE Award W911NF-09-1-0555 and MURI Award 58153-MA-MUR; by the UK EPSRC under Mathematical Sciences Leadership Fellowship EP/I005250/1 and Established Career Fellowship EP/K005413/1, and Institutional Sponsorship Award EP/K503459/1; by the UK Royal Society under a Wolfson Research Merit Award; and by Marie Curie FP7 Integration Grant PCIG12-GA-2012-334622 within the 7th European Union Framework Program.

\bibliographystyle{imsart-nameyear}
\bibliography{networks}

\providecommand*\hyphen{-}
\begin{thebibliography}{30}

\bibitem[\protect\citeauthoryear{Barbour, Holst and
  Janson}{1992}]{barbour1992poisson}
\begin{bbook}[author]
\bauthor{\bsnm{Barbour},~\bfnm{A.~D.}\binits{A.~D.}},
  \bauthor{\bsnm{Holst},~\bfnm{L.}\binits{L.}} \AND
  \bauthor{\bsnm{Janson},~\bfnm{S.}\binits{S.}}
(\byear{1992}).
\btitle{Poisson Approximation}.
\bpublisher{Oxford University Press}, \baddress{Oxford, UK}.
\end{bbook}
\endbibitem

\bibitem[\protect\citeauthoryear{Bickel and
  Chen}{2009}]{bickel2009nonparametric}
\begin{barticle}[author]
\bauthor{\bsnm{Bickel},~\bfnm{P.~J.}\binits{P.~J.}} \AND
  \bauthor{\bsnm{Chen},~\bfnm{A.}\binits{A.}}
(\byear{2009}).
\btitle{A nonparametric view of network models and {N}ewman--{G}irvan and other
  modularities}.
\bjournal{Proc. Natl. Acad. Sci. USA}
\bvolume{106}
\bpages{21068--21073}.
\end{barticle}
\endbibitem

\bibitem[\protect\citeauthoryear{Bickel, Chen and Levina}{2012}]{BickelLevina}
\begin{barticle}[author]
\bauthor{\bsnm{Bickel},~\bfnm{P.~J.}\binits{P.~J.}},
  \bauthor{\bsnm{Chen},~\bfnm{A.}\binits{A.}} \AND
  \bauthor{\bsnm{Levina},~\bfnm{E.}\binits{E.}}
(\byear{2012}).
\btitle{The method of moments and degree distributions for network models}.
\bjournal{Ann. Statist.}
\bvolume{39}
\bpages{2280--2301}.
\end{barticle}
\endbibitem

\bibitem[\protect\citeauthoryear{Bollob{\'a}s, Janson and
  Riordan}{2007}]{bollobas2007phase}
\begin{barticle}[author]
\bauthor{\bsnm{Bollob{\'a}s},~\bfnm{B.}\binits{B.}},
  \bauthor{\bsnm{Janson},~\bfnm{S.}\binits{S.}} \AND
  \bauthor{\bsnm{Riordan},~\bfnm{O.}\binits{O.}}
(\byear{2007}).
\btitle{The phase transition in inhomogeneous random graphs}.
\bjournal{Random Structures Algorithms}
\bvolume{31}
\bpages{3--122}.
\end{barticle}
\endbibitem

\bibitem[\protect\citeauthoryear{Bollob{\'a}s and
  Riordan}{2009}]{bollobas2009metric}
\begin{bincollection}[author]
\bauthor{\bsnm{Bollob{\'a}s},~\bfnm{B.}\binits{B.}} \AND
  \bauthor{\bsnm{Riordan},~\bfnm{O.}\binits{O.}}
(\byear{2009}).
\btitle{Metrics for sparse graphs}.
In \bbooktitle{Surveys in Combinatorics 2009}
(\beditor{\bfnm{S.}\binits{S.}~\bsnm{Huczynska}},
  \beditor{\bfnm{J.~D.}\binits{J.~D.}~\bsnm{Mitchell}} \AND
  \beditor{\bfnm{C.~M.}\binits{C.~M.}~\bsnm{Roney-Dougal}}, eds.)
\bpages{211--287}.
\bpublisher{Cambridge University Press}, \baddress{Cambridge, UK}.
\end{bincollection}
\endbibitem

\bibitem[\protect\citeauthoryear{Britton, Deijfen and
  Martin-L{\"o}f}{2006}]{britton2006generating}
\begin{barticle}[author]
\bauthor{\bsnm{Britton},~\bfnm{T.}\binits{T.}},
  \bauthor{\bsnm{Deijfen},~\bfnm{M.}\binits{M.}} \AND
  \bauthor{\bsnm{Martin-L{\"o}f},~\bfnm{A.}\binits{A.}}
(\byear{2006}).
\btitle{Generating simple random graphs with prescribed degree distribution}.
\bjournal{J. Stat. Phys.}
\bvolume{124}
\bpages{1377--1397}.
\end{barticle}
\endbibitem

\bibitem[\protect\citeauthoryear{Chatterjee, Diaconis and
  Sly}{2011}]{chatterjee2011random}
\begin{barticle}[author]
\bauthor{\bsnm{Chatterjee},~\bfnm{S.}\binits{S.}},
  \bauthor{\bsnm{Diaconis},~\bfnm{P.}\binits{P.}} \AND
  \bauthor{\bsnm{Sly},~\bfnm{A.}\binits{A.}}
(\byear{2011}).
\btitle{Random graphs with a given degree sequence}.
\bjournal{Ann. Appl. Probab..}
\bvolume{21}
\bpages{1400--1435}.
\end{barticle}
\endbibitem

\bibitem[\protect\citeauthoryear{Chung and Lu}{2002}]{Chung2002pnas}
\begin{barticle}[author]
\bauthor{\bsnm{Chung},~\bfnm{F.}\binits{F.}} \AND
  \bauthor{\bsnm{Lu},~\bfnm{L.}\binits{L.}}
(\byear{2002}).
\btitle{The average distances in random graphs with given expected degrees}.
\bjournal{Proc. Natl. Acad. Sci. USA}
\bvolume{99}
\bpages{15879--15882}.
\end{barticle}
\endbibitem

\bibitem[\protect\citeauthoryear{Chung, Lu and Vu}{2003}]{chung2003eigenvalues}
\begin{barticle}[author]
\bauthor{\bsnm{Chung},~\bfnm{F.}\binits{F.}},
  \bauthor{\bsnm{Lu},~\bfnm{L.}\binits{L.}} \AND
  \bauthor{\bsnm{Vu},~\bfnm{V.}\binits{V.}}
(\byear{2003}).
\btitle{Eigenvalues of random power law graphs}.
\bjournal{Ann. Comb.}
\bvolume{7}
\bpages{21--33}.
\end{barticle}
\endbibitem

\bibitem[\protect\citeauthoryear{Diaconis}{1977}]{diaconis1977finite}
\begin{barticle}[author]
\bauthor{\bsnm{Diaconis},~\bfnm{P.}\binits{P.}}
(\byear{1977}).
\btitle{Finite forms of de~{F}inetti's theorem on exchangeability}.
\bjournal{Synthese}
\bvolume{36}
\bpages{271--281}.
\end{barticle}
\endbibitem

\bibitem[\protect\citeauthoryear{Diaconis and Janson}{2008}]{janson2008graph}
\begin{barticle}[author]
\bauthor{\bsnm{Diaconis},~\bfnm{P.}\binits{P.}} \AND
  \bauthor{\bsnm{Janson},~\bfnm{S.}\binits{S.}}
(\byear{2008}).
\btitle{Graph limits and exchangeable random graphs}.
\bjournal{Rend. Mat. Appl.}
\bvolume{28}
\bpages{33--61}.
\end{barticle}
\endbibitem

\bibitem[\protect\citeauthoryear{Durrett}{2007}]{Durrett}
\begin{bbook}[author]
\bauthor{\bsnm{Durrett},~\bfnm{R.}\binits{R.}}
(\byear{2007}).
\btitle{Random Graph Dynamics}.
\bpublisher{Cambridge University Press}, \baddress{Cambridge, UK}.
\end{bbook}
\endbibitem

\bibitem[\protect\citeauthoryear{Fienberg}{2012}]{fienberg2012brief}
\begin{barticle}[author]
\bauthor{\bsnm{Fienberg},~\bfnm{Stephen~E}\binits{S.~E.}}
(\byear{2012}).
\btitle{A brief history of statistical models for network analysis and open
  challenges}.
\bjournal{J. Comput. Graph. Statist.}
\bvolume{21}
\bpages{825--839}.
\end{barticle}
\endbibitem

\bibitem[\protect\citeauthoryear{Fienberg and
  Rinaldo}{2012}]{fienberg2012maximum}
\begin{barticle}[author]
\bauthor{\bsnm{Fienberg},~\bfnm{S.~E.}\binits{S.~E.}} \AND
  \bauthor{\bsnm{Rinaldo},~\bfnm{A.}\binits{A.}}
(\byear{2012}).
\btitle{Maximum likelihood estimation in log-linear models}.
\bjournal{Ann. Statist.}
\bvolume{40}
\bpages{996--1023}.
\end{barticle}
\endbibitem

\bibitem[\protect\citeauthoryear{Hald}{1968}]{hald1968mixed}
\begin{barticle}[author]
\bauthor{\bsnm{Hald},~\bfnm{A.}\binits{A.}}
(\byear{1968}).
\btitle{The mixed {B}inomial distribution and the posterior distribution of $p$
  for a continuous prior distribution}.
\bjournal{J. R. Stat. Soc. Ser. B Stat. Methodol.}
\bpages{359--367}.
\end{barticle}
\endbibitem

\bibitem[\protect\citeauthoryear{Hoeffding}{1963}]{hoeffding1963probability}
\begin{barticle}[author]
\bauthor{\bsnm{Hoeffding},~\bfnm{W.}\binits{W.}}
(\byear{1963}).
\btitle{Probability inequalities for sums of bounded random variables}.
\bjournal{J. Amer. Statist. Assoc.}
\bvolume{58}
\bpages{13--30}.
\end{barticle}
\endbibitem

\bibitem[\protect\citeauthoryear{Holland and
  Leinhardt}{1981}]{holland81exponential}
\begin{barticle}[author]
\bauthor{\bsnm{Holland},~\bfnm{P.~W.}\binits{P.~W.}} \AND
  \bauthor{\bsnm{Leinhardt},~\bfnm{S.}\binits{S.}}
(\byear{1981}).
\btitle{An exponential family of probability distributions for directed graphs
  (with discussion)}.
\bjournal{J. Amer. Statist. Assoc.}
\bvolume{76}
\bpages{33--50}.
\end{barticle}
\endbibitem

\bibitem[\protect\citeauthoryear{Janson}{2010}]{janson2010asymptotic}
\begin{barticle}[author]
\bauthor{\bsnm{Janson},~\bfnm{S.}\binits{S.}}
(\byear{2010}).
\btitle{Asymptotic equivalence and contiguity of some random graphs}.
\bjournal{Random Structures Algorithms}
\bvolume{36}
\bpages{26--45}.
\end{barticle}
\endbibitem

\bibitem[\protect\citeauthoryear{Molloy and Reed}{1995}]{molloy1995critical}
\begin{barticle}[author]
\bauthor{\bsnm{Molloy},~\bfnm{M.}\binits{M.}} \AND
  \bauthor{\bsnm{Reed},~\bfnm{B.}\binits{B.}}
(\byear{1995}).
\btitle{A critical point for random graphs with a given degree sequence}.
\bjournal{Random Structures Algorithms}
\bvolume{6}
\bpages{161--180}.
\end{barticle}
\endbibitem

\bibitem[\protect\citeauthoryear{Newman}{2001}]{newman2001structure}
\begin{barticle}[author]
\bauthor{\bsnm{Newman},~\bfnm{M.~E.~J.}\binits{M.~E.~J.}}
(\byear{2001}).
\btitle{The structure of scientific collaboration networks}.
\bjournal{Proc. Natl. Acad. Sci. USA}
\bvolume{98}
\bpages{404--409}.
\end{barticle}
\endbibitem

\bibitem[\protect\citeauthoryear{Newman, Watts and
  Strogatz}{2002}]{newman2002random}
\begin{barticle}[author]
\bauthor{\bsnm{Newman},~\bfnm{M.~E.~J.}\binits{M.~E.~J.}},
  \bauthor{\bsnm{Watts},~\bfnm{D.~J.}\binits{D.~J.}} \AND
  \bauthor{\bsnm{Strogatz},~\bfnm{S.~H.}\binits{S.~H.}}
(\byear{2002}).
\btitle{Random graph models of social networks}.
\bjournal{Proc. Natl. Acad. Sci. USA}
\bvolume{99}
\bpages{2566--2572}.
\end{barticle}
\endbibitem

\bibitem[\protect\citeauthoryear{Perry and Wolfe}{2012}]{PerryWolfe2012a}
\begin{bunpublished}[author]
\bauthor{\bsnm{Perry},~\bfnm{P.~O.}\binits{P.~O.}} \AND
  \bauthor{\bsnm{Wolfe},~\bfnm{P.~J.}\binits{P.~J.}}
(\byear{2012}).
\btitle{Null models for network data}.
\bnote{Unpublished manuscript (arXiv:1201.5871)}.
\end{bunpublished}
\endbibitem

\bibitem[\protect\citeauthoryear{Rinaldo, Petrovi\'c and
  Fienberg}{2013}]{rinaldo2013maximum}
\begin{barticle}[author]
\bauthor{\bsnm{Rinaldo},~\bfnm{A.}\binits{A.}},
  \bauthor{\bsnm{Petrovi\'c},~\bfnm{S.}\binits{S.}} \AND
  \bauthor{\bsnm{Fienberg},~\bfnm{S.~E.}\binits{S.~E.}}
(\byear{2013}).
\btitle{Maximum likelihood estimation in the {B}eta model}.
\bjournal{Ann. Statist.}
\bnote{In press (arXiv:1105.6145)}.
\end{barticle}
\endbibitem

\bibitem[\protect\citeauthoryear{Robinson}{1994}]{robinson1994semiparametric}
\begin{barticle}[author]
\bauthor{\bsnm{Robinson},~\bfnm{P.~M.}\binits{P.~M.}}
(\byear{1994}).
\btitle{Semiparametric analysis of long-memory time series}.
\bjournal{Ann. Statist.}
\bvolume{22}
\bpages{515--539}.
\end{barticle}
\endbibitem

\bibitem[\protect\citeauthoryear{Rohe, Chatterjee and
  Yu}{2011}]{rohe2011spectral}
\begin{barticle}[author]
\bauthor{\bsnm{Rohe},~\bfnm{K.}\binits{K.}},
  \bauthor{\bsnm{Chatterjee},~\bfnm{S.}\binits{S.}} \AND
  \bauthor{\bsnm{Yu},~\bfnm{B.}\binits{B.}}
(\byear{2011}).
\btitle{Spectral clustering and the high-dimensional stochastic blockmodel}.
\bjournal{Ann. Statist.}
\bvolume{39}
\bpages{1878--1915}.
\end{barticle}
\endbibitem

\bibitem[\protect\citeauthoryear{Sussman, Tang and
  Priebe}{2013}]{sussman2013universally}
\begin{barticle}[author]
\bauthor{\bsnm{Sussman},~\bfnm{D.~L.}\binits{D.~L.}},
  \bauthor{\bsnm{Tang},~\bfnm{M.}\binits{M.}} \AND
  \bauthor{\bsnm{Priebe},~\bfnm{C.~E.}\binits{C.~E.}}
(\byear{2013}).
\btitle{Universally consistent latent position estimation and vertex
  classification for random dot product graphs}.
\bjournal{Ann. Statist.}
\bnote{In press (arXiv:1207.6745)}.
\end{barticle}
\endbibitem

\bibitem[\protect\citeauthoryear{Tricomi and
  Erd{\'e}lyi}{1951}]{tricomi1951asymptotic}
\begin{barticle}[author]
\bauthor{\bsnm{Tricomi},~\bfnm{F.~G.}\binits{F.~G.}} \AND
  \bauthor{\bsnm{Erd{\'e}lyi},~\bfnm{A.}\binits{A.}}
(\byear{1951}).
\btitle{The asymptotic expansion of a ratio of {G}amma functions}.
\bjournal{Pacific J. Math.}
\bvolume{1}
\bpages{133--142}.
\end{barticle}
\endbibitem

\bibitem[\protect\citeauthoryear{van~der Hofstad}{2013}]{van2012critical}
\begin{barticle}[author]
\bauthor{\bparticle{van~der} \bsnm{Hofstad},~\bfnm{R.}\binits{R.}}
(\byear{2013}).
\btitle{Critical behavior in inhomogeneous random graphs}.
\bjournal{Random Structures Algorithms}
\bvolume{42}
\bpages{480--508}.
\end{barticle}
\endbibitem

\bibitem[\protect\citeauthoryear{Whittaker and
  Watson}{1927}]{whittaker1927course}
\begin{bbook}[author]
\bauthor{\bsnm{Whittaker},~\bfnm{E.~T.}\binits{E.~T.}} \AND
  \bauthor{\bsnm{Watson},~\bfnm{G.~N.}\binits{G.~N.}}
(\byear{1927}).
\btitle{A Course of Modern Analysis}, \bedition{Fourth} ed.
\bpublisher{Cambridge University Press}, \baddress{Cambridge, UK}.
\end{bbook}
\endbibitem

\bibitem[\protect\citeauthoryear{Zhao, Levina and
  Zhu}{2012}]{zhao2012consistency}
\begin{barticle}[author]
\bauthor{\bsnm{Zhao},~\bfnm{Y.}\binits{Y.}},
  \bauthor{\bsnm{Levina},~\bfnm{E.}\binits{E.}} \AND
  \bauthor{\bsnm{Zhu},~\bfnm{J.}\binits{J.}}
(\byear{2012}).
\btitle{Consistency of community detection in networks under degree-corrected
  stochastic block models}.
\bjournal{Ann. Statist.}
\bvolume{40}
\bpages{2266--2292}.
\end{barticle}
\endbibitem

\end{thebibliography}

\end{document}